\documentclass[a4paper]{article}
\usepackage[all]{xy}\usepackage[latin1]{inputenc}        
\usepackage[dvips]{graphics,graphicx}
\usepackage{amsfonts,amssymb,amsmath,color,mathrsfs, amstext}
\usepackage{amsbsy, amsopn, amscd, amsxtra, amsthm,authblk}
\usepackage{enumerate,algorithmicx,algorithm}
\usepackage{algpseudocode}
\usepackage{upref}
\usepackage{geometry}
\usepackage[displaymath]{lineno}
\usepackage{float}
\usepackage{bm}
\usepackage{dsfont}
\usepackage{caption}
\usepackage{yhmath}
\usepackage{booktabs}
\usepackage{subcaption}
\usepackage{mathtools}
\usepackage{makecell}
\usepackage[colorlinks,
            linkcolor=red,
            anchorcolor=red,
            citecolor=red
            ]{hyperref}
\usepackage{pdfpages}
\usepackage{multirow}
\usepackage{ulem}

\numberwithin{equation}{section}

\newcommand{\bbE}{\mathbb{E}}
\newcommand{\bbR}{\mathbb{R}}
\newcommand{\bbP}{\mathbb{P}}

\newcommand{\bbZ}{\mathbb{Z}}

\newcommand{\cN}{\mathcal{N}}
\newcommand{\cO}{\mathcal{O}}
\newcommand{\cS}{\mathcal{S}}
\newcommand{\cF}{\mathcal{F}}

\newcommand{\cJ}{\mathcal{J}}

\newcommand{\tbx}{\tilde{\bm{x}}}
\newcommand{\tbp}{\tilde{\bm{p}}}

\def\hat{\widehat}
\def\tilde{\widetilde}
\def\HMC{\mathsf{HMC}}

\newtheorem{theorem}{Theorem}[section]
\newtheorem{lemma}{Lemma}[section]

\newtheorem{remark}{Remark}[section]

\newtheorem{assumption}{Assumption}[section]

\numberwithin{equation}{section}

\title{A splitting Hamiltonian Monte Carlo method for efficient sampling}
\date{}

\author[1,2,3]{Lei Li\thanks{E-mail: leili2010@sjtu.edu.cn}}
\author[1,2,3,4]{Lin Liu\thanks{E-mail: linliu@sjtu.edu.cn}}
\author[1]{Yuzhou Peng\thanks{E-mail: pengyz@sjtu.edu.cn}}
\affil[1]{School of Mathematical Sciences, Shanghai Jiao Tong University, Shanghai, 200240, P. R. China}
\affil[2]{Institute of Natural Sciences and MOE-LSC, Shanghai Jiao Tong University, Shanghai, 200240, P. R. China}
\affil[3]{Shanghai Artificial Intelligence Laboratory, Shanghai, P. R. China}
\affil[4]{CMA-Shanghai and SJTU-Yale Joint Center for Biostatistics and Data Science, Shanghai Jiao Tong University, Shanghai, 200240, P. R. China}
\begin{document}

\maketitle

\begin{abstract}
We propose a splitting Hamiltonian Monte Carlo (SHMC) algorithm, which can be computationally efficient when combined with the random mini-batch strategy. 
By splitting the potential energy into numerically nonstiff and stiff parts, one makes a proposal using the nonstiff part of $U$, followed by a Metropolis rejection step using the stiff part that is often easy to compute. The splitting allows efficient sampling from systems with singular potentials (or distributions with degenerate points) and/or with multiple potential barriers.
In our SHMC algorithm, the proposal only based on the nonstiff part in the splitting is generated by the Hamiltonian dynamics, which can be potentially more efficient than the overdamped Langevin dynamics. We also use random batch strategies to reduce the computational cost to $\mathcal{O}(1)$ per time step in generating the proposals for problems arising from many-body systems and Bayesian inference, and prove that the errors of the Hamiltonian induced by the random batch approximation is $\mathcal{O}(\sqrt{\Delta t})$ in the strong and $\mathcal{O}(\Delta t)$ in the weak sense, where $\Delta t$ is the time step. Numerical experiments are conducted to verify the theoretical results and the computational efficiency of the proposed algorithms in practice.  
\end{abstract}

\allowdisplaybreaks
\section{Introduction}
\label{sec:intro}

Markov chain Monte Carlo (MCMC) \cite{roberts2004general, metropolis1953equation, hastings1970monte, geman1984stochastic, swendsen1986replica, benettin1994hamiltonian} methods are nowadays routinely used in a variety of scientific computing problems, including computing statistics for many-body systems \cite{allen1987, frenkel2001understanding}, sampling from log-concave distributions \cite{lovasz1993random, dalalyan2017theoretical, lee2018convergence, chen2020fast}, parameter estimation in Bayesian statistics \cite{geman1984stochastic, banerjee2014hierarchical, green1995reversible,smith2013sequential} and Bayesian inverse problems \cite{giordano2020consistency, nickl2020polynomial}, just to name a few. Among MCMC methods, Hamiltonian Monte Carlo (HMC) \cite{duane1987hybrid, neal2011mcmc, betancourt2017conceptual, barbu2020hamiltonian} has recently garnered a lot of attention in practice due to its scalability and efficiency in high-dimensional settings \cite{neal2011mcmc, beskos2013optimal}. Nonetheless, there are several situations where HMC can encounter difficulties. The first such scenario might be sampling from the Gibbs distribution
\begin{gather}\label{eq:Gibbs-U}
\mu(q) \propto \exp\left[-\beta U(q)\right]
\end{gather}
of a many-body interacting particle system, where $\beta > 0$ is the (dimensionless) inverse temperature. It takes $\cO(N^2)$ operations to compute the total potential energy
\begin{gather}\label{eq:manybodypotential}
U(q) = \sum_{i = 1}^N w_i V(q_i) + \sum_{i,j:i<j} w_i w_j \phi(q_i-q_j), 
\end{gather}
where $q_i\in\bbR^d$ is the position of the $i^{th}$ particle and $w_i$ denotes the weight.
If one moves one particle per step, which is preferred in some applications \cite{frenkel2001understanding,li2020random, ding2021random}, the computational cost of evolving the Hamiltonian dynamics and the Metropolis-Hastings correction step in HMC are both $\cO(N)$. This fact makes HMC computationally expensive when sampling from the Gibbs distribution. 
Moreover, interaction potentials $\phi$ such as the Coulomb potential or the Lennard-Jones potential are usually singular \cite{frenkel2001understanding}. Singularity in $\phi$ can introduce stiffness to the Hamiltonian system, which makes the numerical simulations difficult \cite{wanner1996solving}, and possibly leads to low acceptance rates \cite{li2020random}, thus deteriorating the sampling efficiency of HMC. 

As for another well-known example, let us consider Bayesian inference of a parameter $\theta$ based on its posterior distribution $p_{\mathrm{post}} (\theta | \mathcal{D}_N)$ given the observed data $\mathcal{D}_{N}$
\begin{gather}\label{eq:post}
p_{\mathrm{post}} (\theta | \mathcal{D}_N)  \propto p_{\mathrm{prior}} (\theta) \prod\limits_{i = 1}^N p (y_i; \theta),
\end{gather}
where $\mathcal{D}_N = \{y_1, \ldots, y_N\}$ is a sample of size $N$ drawn from the probabilistic model $p (\cdot; \theta)$. 
When $p (\cdot; \theta)$ is posited to be a mixture model, which is often the case in clustering and density estimation problems \cite{fraley2002model}, the corresponding posterior distribution $p_{\mathrm{post}} (\theta | \mathcal{D}_N)$ may be multimodal. One main reason for multimodality in mixture models is the nonidentifiability of the parameters due to label switching \cite{jasra2005markov}. When the target distribution (the posterior distribution in this case) is multimodal, most MCMC algorithms have difficulty moving between isolated modes and are therefore prone to generate biased samples. As a result, HMC could fail to explore the entire state space and lead to even worse performance than the simple Random Walk Metropolis (RWM) algorithm \cite{mangoubi2018does}. Indeed, the Hamiltonian simulation heavily relies on the gradient information of the potential $U$, easily resulting in samples trapped in one single well of $U$ when $U$ has multiple modes.
A popular strategy of sampling from multimodal distributions is running multiple chains with over-dispersed initializations in parallel. Other methods include parallel tempering \cite{geyer1991markov, miasojedow2013adaptive}, simulated tempering \cite{marinari1992simulated}, Wormhole HMC \cite{lan2014wormhole} and the Wang-Landau algorithm \cite{wang2001determining, wang2001efficient, dai2020wang}. 

In order to deal with the numerical stiffness in the potential for the Gibbs distributions of many-body particle systems and multi-modal distributions, like the singularity in Lennard Jones potential and the deep wells in the multi-modal distributions mentioned above, we adopt the idea of potential splitting  \cite{sexton1992hamiltonian,hetenyi2002multiple,neal2011mcmc, shahbaba2014split,li2020random} for this numerical challenge. Compared to the existing potential splitting for HMC \cite{neal2011mcmc, shahbaba2014split}, our motivation is different in the sense that we aim to separate out the numerically stiff part so that some standard methods can be used for generating the proposal samples. 
This motivation is similar to the Random Batch Monte Carlo (RBMC) method \cite{li2020random}, but we choose to use the Hamiltonian dynamics to generate the samples for possible shorter mixing time. This leads us to a novel splitting Hamiltonian Monte Carlo (SHMC) method. In particular, consider a target distribution $\rho$ with the effective potential energy:
\begin{gather}\label{eq:totalenergy}
U = -\beta^{-1}\log \rho + const.
\end{gather}
After splitting the potential $U$ into two parts $U = U_1 + U_2$, a candidate state is first proposed by evolving the Hamiltonian system corresponding to $U_1$ and then is probabilistically accepted by a Metropolis rejection step determined by $U_2$. 
By combining with the random mini-batch strategy, the SHMC method could address the above issues regarding computational costs, singularity and multimodality to some extent.
The benefits of our SHMC method can be summarized as follows:
\begin{enumerate}[(a)]
\item  When sampling from the Gibbs distribution \eqref{eq:Gibbs-U} of an interacting particle system, particularly with singular interaction, we decompose $U$ into a nonstiff part $U_1$ and a part $U_2$ with short-range interactions.
The nonstiffness of $U_1$ guarantees the feasibility of the Hamiltonian simulation while $U_2$ reduces the computational cost of the Metropolis correction step since only particles within a small neighborhood need to be considered when calculating the acceptance probability. 

\item There is often a big summation term in the potential energy \eqref{eq:totalenergy}; e.g. see equation \eqref{eq:manybodypotential}. In this paper, we apply the random mini-batch idea \cite{robbins1951stochastic, jin2018random} to further reduce the computation load due to summing over many terms, leading to more efficient algorithms. In particular, if only one randomly picked particle is updated \cite{frenkel2001understanding} in each iteration, the cost can be reduced to $\cO(1)$ per iteration. 

\item As for a multimodal distribution, we choose $U_1$ with a relatively flat landscape to enable the Markov chains traveling between isolated wells. Such $U_1$ can be constructed by either reducing the height of the barriers between different wells of $U$ or filling the wells with ``computational sand'', an idea borrowed from the metadynamics approach \cite{laio2002escaping} (for further details, see Remark \ref{rem:meta}).

\item As we shall see later, our proposed SHMC method can possibly have shorter mixing time compared to the Langevin-based methods like the RBMC method.
\end{enumerate}

We remark that the random batch strategy has recently been applied to interacting particle systems, resulting in several efficient algorithms \cite{jin2018random,li2020random, li2020stochastic, jin2020randomQMC, jin2020randombatchEwald}. 
Different from the mini-batch in stochastic gradient descent (SGD) \cite{robbins1951stochastic, bottou1998online, bubeck2015convex} and stochastic gradient Hamiltonian Monte Carlo (SGHMC) \cite{chen2014stochastic}, the random batch methods for interacting particle systems aim to grasp the dynamics as well. The rationale is a time averaging effect, by which RBM methods converge in a ``law of large number'' manner in time \cite{jin2018random, jin2021random}. Introducing random batch strategy into the Hamiltonian dynamics, however, will introduce extra effective noise, and the Hamiltonian is no longer conserved and long time simulation is not accurate.
Nonetheless, in our proposed method, the Hamiltonian dynamics is never performed for a long time so this issue is nonessential in this paper; see Section \ref{subsec:error} for the error estimates.

The rest of this paper is organized as follows. 
In Section \ref{sec:SHMC}, we propose our algorithm and provide some motivating applications. 
Specifically, we first elaborate the general framework of the splitting Monte Carlo method. Then we incorporate HMC into the general framework and present the splitting Hamiltonian Monte Carlo (SHMC) algorithm in Section \ref{subsec:SHMC}. 
Finally, taking interacting particle systems and Bayesian inference problems as examples, we further introduce the random batch idea to reduce computational cost in Section \ref{subsec:RB-SHMC}. 
In Section \ref{sec:Theory&discussion}, we prove an error bound of the value of the approximate Hamiltonian generated in random batches and show the advantage of our method over sampling algorithms based on overdamped Langevin equations. 
Finally, we perform some numerical experiments to verify the high numerical and computational efficiency of our algorithm in Section \ref{sec:Numerical examples}.

\section{The Splitting Hamiltonian Monte Carlo method and its random batch variant}\label{sec:SHMC}

In this section, we propose the Splitting Hamiltonian Monte Carlo (SHMC) method to overcome the aforementioned difficulties that HMC could encounter in practice. To make SHMC scalable to large-scale problems such as many-body systems and Bayesian inference, we also incorporate the idea of random mini-batch into SHMC. We abbreviate the resulting algorithms {\it RB-SHMC}.

To start with, let us introduce the general framework of the splitting Monte Carlo method; see Algorithm \ref{alg:generalSMC}.
As mentioned in Section \ref{sec:intro}, common MCMC methods such as RWM and HMC behave poorly when sampling from the Gibbs distribution of interacting particle systems with singular potentials and sampling from multimodal distributions that often arise in Bayesian inference of mixture models. 
In both cases, there is some stiff part in the potential $U$, so the traditional gradient-based MCMC methods fail to sample from the target distribution efficiently. 
To alleviate this problem, we first construct a surrogate potential $U_1$ without stiffness and apply
 some generic sampling method $\cS$ satisfying detailed balance with respect to (w.r.t.) the Gibbs distribution $\rho_1 \propto \exp(-\beta U_1)$. Then one obtains a proposed state which is a sample from $\rho_1$ and may be severely biased from the target distribution $\rho \propto \exp(-\beta U)$. 
In the second step, a Metropolis rejection step using $U_2:=U-U_1$ is required to determine whether the aforementioned proposal shall be accepted as a new sample, thereby correcting the bias introduced in the first step. 

The splitting Monte Carlo (Algorithm \ref{alg:generalSMC}) can be viewed as a special case of the Metropolis-Hastings algorithm. Denote the transition probability of $\cS$  from state $x$ to state $y$ as $q(y; x)$. The acceptance probability in the Metropolis rejection step is then given by: 
\begin{gather}\label{eq:acc prob}
A(y; x) := \min\left(1, \frac{\exp[-\beta U(y)]q(x; y)}{\exp[-\beta U(x)]q(y; x)}\right),
\end{gather}
and the detailed balance condition of $\cS$ yields: 
\begin{gather}\label{eq:q(x; y)/q(y; x)}
\frac{q(x; y)}{q(y; x)} = \frac{\exp[-\beta U_1(x)]}{\exp[-\beta U_1(y)]}.
\end{gather}
Substitute identity \eqref{eq:q(x; y)/q(y; x)} into \eqref{eq:acc prob}, one can see that the acceptance probability is only related to the $U_2$ component of $U$:
\begin{gather}
\begin{split}
A(y; x) & = \min\left(1, \frac{\exp[-\beta U(y)]\exp[-\beta U_1(x)]}{\exp[-\beta U(x)]\exp[-\beta U_1(y)]}\right) \\
 & = \min\left(1, \exp\big[-\beta\big(U_2(y) - U_2(x)\big)\big]\right).
\end{split}
\end{gather}

\begin{algorithm}[H]
\caption{(General Splitting Monte Carlo algorithm)}\label{alg:generalSMC}
\begin{algorithmic}[1]
\State \textbf{Step 1} --- Propose a candidate state $x^*$ using some generic sampling method $\cS$ satisfying detailed balance w.r.t. the Gibbs distribution corresponding to $U_{1}$:
\begin{gather}
\exp\left[-\beta U_1(x)\right] q(x; x^*)= \exp\left[-\beta U_1(x^*)\right] q(x^*; x).
\end{gather}

\State \textbf{Step 2} --- Set $x \leftarrow x^*$ with probability:
\begin{gather}
A(x^*; x) := \min\left(1, \exp\left[-\beta\big(U_2(x^*) - U_2(x)\big)\right]\right).
\end{gather}
Otherwise, $x \leftarrow x$ remains unchanged. 
\end{algorithmic}
\end{algorithm}

\subsection{Splitting Hamiltonian Monte Carlo method}\label{subsec:SHMC}

We now present a splitting Monte Carlo method based on HMC, coined as the Splitting Hamiltonian Monte Carlo (SHMC) method, which evolves a Hamiltonian system to make a proposal in \textbf{Step 1} of Algorithm \ref{alg:generalSMC}. 

\begin{algorithm}[H]
\caption{(Splitting Hamiltonian Monte Carlo algorithm)}\label{alg:SHMC}
\begin{algorithmic}[1]
\State Split $U:=U_1+U_2$ such that $U_1$ is nonstiff and the remaining part $U_2$ contains the numerically stiff part of $U$ and can be computed efficiently.
Randomly generate the initial position $x^{(0)}$ and set $N_s$ as the total number of samples.
\For {$n = 1, \cdots, N_s$}
	\State Randomly pick $x_i$ with uniform probability. Sample a momentum $p_i \sim \cN\left(0, \frac{m}{\beta} I_d\right)$ and set $L_n \geq 1, \Delta t_n > 0$.
	\State $(x_i^*, p_i^*) \leftarrow (x_i^{(n-1)}, p_i)$. 
	\For {$\ell = 1,\cdots, L_n$}
		\State Perform the following leapfrog discretization
			\begin{gather}\label{eq:SHMC 2-L}
 			\begin{split}
 				& p_i^* \leftarrow p_i^* - \frac{\Delta t_n}{2} \nabla U_1(x_i^*); \\
 				& x_i^* \leftarrow x_i^* + \Delta t_n \frac{p_i^*}{m}; \\
 				& p_i^* \leftarrow p_i^* - \frac{\Delta t_n}{2} \nabla U_1(x_i^*).
 			\end{split}
 			\end{gather} 	
	\EndFor 
	 \State Evaluate the following Metropolis acceptance probability:
     	\[
      		A\leftarrow A(x_i^*;x_i)=\min\left(1, \exp\left[-\beta\left(U_2(x_i^*)-U_2(x_i^{(n-1)})\right)\right]\right).
      	\]
      \State Generate a random number $\zeta$ from uniform distribution on $[0, 1]$. If $\zeta\le A$, set
      	\[
      		x_i^{(n)} \leftarrow x_i^*.
      	\]
      	Otherwise, set
      	\[
      		x_i^{(n)} \leftarrow x_i^{(n-1)}.
      	\]
\EndFor  
\end{algorithmic}
\end{algorithm}

We first recall some fundamental properties of the standard HMC, which obviously satisfies the detailed balance condition. HMC generates samples from the $D$-dimensional target distribution \eqref{eq:Gibbs-U} by introducing an auxiliary variable $p \in \bbR^D$, namely the momentum, and sampling from the joint distribution $\tilde{\mu}(x, p) \propto \exp[-\beta H(x, p)]$. Traditionally, the Hamiltonian $H(x, p) = U(x) + K(p)$ is separable with $K(p) = \|p\|^2/(2m)$, where $m$ can be interpreted as ``mass'' in physics term.
Hence, $p$ is sampled from the isotropic multivariate Gaussian distribution $\cN\left(0,\frac{m}{\beta} I_D\right)$.
The Hamiltonian system satisfies many nice properties such as conservation of energy, reversibility, and symplecticity \cite{barbu2020hamiltonian, neal2011mcmc}, and therefore it is usually discretized by the leapfrog method, a symmetric symplectic integrator \cite{leimkuhler2004simulating}. 
HMC proposes a new state in the phase space by simulating a Hamiltonian dynamics. 
The proposed momentum is flipped after the numerical simulation to guarantee the time reversibility and a correction step follows to correct the discretization error induced by the leapfrog integrator. 
In practice, the momentum flipping step is skipped since $K$ is symmetric and the momentum should be resampled to ensure ergodicity. 
However, frequent resampling may result in random walk and thus multiple leapfrog steps are performed per iteration. 

Sharing the same algorithmic structure as HMC but differing in some specific steps, SHMC numerically evolves a Hamiltonian system with a surrogate potential $U_1$ to propose a candidate state $(x^*, p^*)$ and carries out a Metropolis rejection step only using $U_2 := U - U_1$ after the momentum flipping step.  
The same as HMC, flipping $p^*$ ensures time reversibility but is usually omitted since the rejection step has nothing to do with the momentum and it should also be resampled in the next iteration to ensure ergodicity. Note that the candidate state $(x^*, p^*)$ is proposed based on the surrogate potential $U_1$ and may not be a typical sample from the target distribution. Therefore, the rejection step for $U_{2}$ is critical to guarantee the empirical distribution of the samples to converge to the target distribution, beyond the sole purpose of correcting the discretization error as in HMC. 

\begin{remark}
\label{rmk:error}
Recall that since the Hamiltonian system is numerically solved by the leapfrog integrator, which is symplectic and has a second-order accuracy, the discretization error would not be a great issue. 
Moreover, the method will have negligible systematic error if the step size $\Delta t_n$ gradually decreases to zero and the Metropolis correction step in HMC to correct the errors introduced by the time discretization can be omitted to reduce computational cost. A related error bound can be found in Theorem \ref{thm:E DeltaH} below. We also refer interested readers to the experiments in Section 5 of \cite{welling2011bayesian} for further empirical evidence that decreasing step size helps to eliminate the systematic error. 
\end{remark}

For high-dimensional problems, especially the interacting particle systems, randomly picking a few entries of $x$ to update per iteration can be more efficient \cite{frenkel2001understanding}. In the following, we assume $x = [x_1, \ldots, x_N] \in \bbR^{d \times N}$ and only $x_i$, one randomly chosen entry of $x$, is updated per iteration. 
The detailed procedure of our SHMC method is given in Algorithm \ref{alg:SHMC} and the splitting schemes are presented in Section \ref{subsec:RB-SHMC} below. 

\subsection{Splitting Hamiltonian Monte Carlo method with random batch}\label{subsec:RB-SHMC}

In this section, we discuss how to apply the random batch idea \cite{robbins1951stochastic} to SHMC for problems with big summation in the potential $U$.
Typical examples include the Gibbs distribution arising from interacting particle systems and the posterior distribution from Bayesian inference problems.
For these examples, any algorithm that needs to evaluate the full summation, such as the traditional Metropolis Hastings algorithm and the standard HMC, will be computationally inefficient.

The random batch idea is originated from the famous SGD algorithm \cite{robbins1951stochastic, bottou1998online} and the stochastic gradient Langevin dynamics (SGLD) for Bayesian inference \cite{welling2011bayesian}. The random mini-batch based on random grouping strategy has recently been applied to many-body interacting systems \cite{jin2018random,li2020stochastic,jin2020randomsecondorder, jin2021convergence, ye2021efficient}. Recently, a novel efficient molecular dynamics simulation method by building random batch and importance sampling into the Fourier space of Ewald sum has been proposed in \cite{jin2020randombatchEwald}.
In a nutshell, the random batch strategy approximates a big summation by summing over only a small random subset of size $\cO(1)$ and therefore the complexity per time step is reduced to $\cO(1)$ \cite{jin2018random, li2020random}. The convergence of the random batch lies in the time average behavior and it can be intuitively interpreted as an extension of the ``law of large number'' from ``over samples'' to ``over time''.

\begin{remark}\leavevmode
\label{rmk:error2}
\begin{enumerate}[(a)]
    \item Introducing random batch strategy into the Hamiltonian dynamics will introduce extra noise and lead to the so-called ``numerical heating'' effects \cite{jin2021random}. The Hamiltonian is no longer conserved and the long-run simulation will not be accurate. However, in either HMC or SHMC, the Hamiltonian dynamics are only performed for a short time so this issue is nonessential; see Section \ref{subsec:error} for error bound estimates. 
    \item The error in the Hamiltonian introduced by random batch, albeit small as proved in Section \ref{subsec:error}, will necessarily introduce systematic errors in the invariant measures of the Markov chain. One may correct this using a Metropolis rejection step as in the Metropolis-adjusted Langevin algorithm (MALA) \cite{benettin1994hamiltonian, roberts1996exponential}. However, we choose not to do this simply because such a correction will bring back an $\cO(N)$ computation cost that we have tried to avoid using random batch. As above in Remark \ref{rmk:error}, one may decrease $\Delta t$ to make the systematic error negligible.
\end{enumerate}
\end{remark}

In the following, we present two implementations of the RB-SHMC algorithm.

\subsubsection{Gibbs distribution of interacting particle system}

For the Gibbs distribution of a $N$-particle interacting system, the potential \eqref{eq:manybodypotential} comprises the confining potential $V$ and the interaction potential $\phi$.
Without loss of generality, we assume that the confining potential is smooth and the interaction potential is symmetric in the sense $\phi(q_i - q_j) = \phi(q_j - q_i)$.
We also assume $w_i \equiv w$ for simplicity. 
In practice, the interaction between particles is of long range and usually singular. 
It is thus recommended that $\phi$ is decomposed into a smooth part $\phi_1$ and a short-range part $\phi_2$. 
Hence, one has
\begin{gather}\label{eq:Usplittingips}
\begin{split}
& U_1(q) = w\sum_{i = 1}^N V(q_i) + w^2\sum_{i,j: i< j} \phi_1(q_i - q_j), \\
& U_2(q) = w^2\sum_{i,j: i< j} \phi_2(q_i - q_j),
\end{split}
\end{gather}
where $q = [q_1, \ldots, q_N]$ with $q_i \in \mathbb{R}^d$ denoting the position of the $i^{th}$ particle and $w$ represents the weight. 
Clearly, the short range nature of $\phi_2$ is essential to reducing the computational cost because only particles within that short range need to be counted and therefore $U_2$ can be efficiently evaluated using suitable data structures like cell lists \cite[Appendix F]{frenkel2001understanding}.

The Hamiltonian dynamics corresponding to $U_1$ defined in \eqref{eq:Usplittingips} are given by
\begin{gather}\label{eq:Hamiltonphi1}
 \begin{split}
  & \dot{q}_i =\frac{p_i}{m}, \\
  & \dot{p}_i = -\left(\frac{\nabla V(q_i)}{w(N-1)} + \frac{1}{N-1}\sum_{j: j\neq i}\nabla \phi_1(q_i-q_j)\right). 
\end{split}
\end{gather}
Note that we have done a time rescaling $t \leftarrow w\sqrt{N-1} t$ and redefined the momentum by $p \leftarrow p/(w\sqrt{N-1})$ so that the interacting force on the $i$-th particle is $\cO(1)$.
The step of resampling of $p$ in Algorithm \ref{alg:SHMC} now becomes
\[
p_i\sim \mathcal{N}\left(0, \frac{m}{w^2(N-1)}I_d \right).
\]
The acceptance rate is again given by $\min(1, \exp(-\Delta U_2))$.
\begin{remark}
Note that the time rescaling here is equivalent to setting $\tilde{U}_1=\sum_{i=1}^N\frac{V(q_i)}{w(N-1)}+\frac{1}{N-1}\sum_{i,j: i<j}\phi_1(q_i-q_j)$ and $\tilde{U}_2=\frac{1}{N-1}\sum_{i,j: i<j}\phi_2(q_i-q_j)$, and choosing $\beta=w^2(N-1)$ in Algorithm \ref{alg:SHMC}. The acceptance rate is then again based on $\beta \Delta \tilde{U}_2=\Delta U_2$.
\end{remark}
In the case of mean-field regime where $w\sim 1/N$, it becomes the usual ODE system for interacting particles in literature.

When random batch is applied, the total interacting force acting on $i$, which should be contributed by all the other particles $j \neq i$, is now approximated by that only from a random batch of particles and the Hamiltonian system is reduced to 
\begin{gather}\label{eq:partially coupled Hamiltonian system}
 \begin{split}
  & \dot{\tilde{q}}_i = \frac{\tilde{p}_i}{m}, \\
  & \dot{\tilde{p}}_i = -\left(\frac{\nabla V(\tilde{q}_i)}{w(N-1)} + \frac{1}{s}\sum_{j\in \xi}\nabla \phi_1(\tilde{q}_i-\tilde{q}_j)\right),
\end{split}
\end{gather}
where $\xi = \{\xi_1, \ldots, \xi_s\}$ is a random subset of $\{1, \ldots, N\} \setminus \{i\}$, bookkeeping the indices of the particles selected into the small batch. 

We remark that there exist other ways to implement the random batch strategy. For example, in the case of Coulomb interactions in a periodic box, the importance sampling in Fourier space is possible and the time scaling in \eqref{eq:Hamiltonphi1} may be unnecessary \cite{jin2020randombatchEwald}. Even if we do random batch in real space as indicated above in the molecular dynamics regime where $w=1$, the time scaling to have $1/N$ factor in \eqref{eq:Usplittingips} is only a convenience made for establishing theoretical results. It will not change the intrinsic dynamics and is not necessary for implementation in practice \cite{jin2020randomsecondorder}.

\subsubsection{Multimodal posterior in Bayesian inference}

In Bayesian inference problems, the potential of the posterior is given by:
\begin{gather}
U(\theta) = \frac{1}{\beta} \left[-\log p_{\mathrm{prior}}(\theta) + \sum_{i = 1}^N -\log p(y_i; \theta)\right],
\end{gather}
where $\theta$ is the parameter of interest, $y_i, i = 1, \ldots, N$ are the observations and $\beta > 0$ is a scaling factor\footnote{Note that the $N$ here denotes the number of observed data and has nothing to do with the dimension of the variable $\theta$.}. 

In many cases, $p_{\mathrm{post}}$ is multimodal and thus there are multiple wells in the landscape of $U$. 
The energy barriers between wells frustrate general MCMC methods and lead to biased samples trapped in one single well.  
In this situation, one can flatten the landscape of $U$ to construct $U_1$ by adding ``computational sand'' $G$ into the wells and thus $U_2 = -G$. One feasible choice of $G$ is the Gaussian kernel.  

Here, we take the potential $U_1$ to be
\[
U_1(\theta)=-\log p_{\mathrm{prior}}(\theta) + \sum_{i = 1}^N -\log p(y_i; \theta)+G(\theta).
\]
As above, we do a time rescaling $t \leftarrow \sqrt{N} t$ and the momentum is redefined as $p \leftarrow p/\sqrt{N}$. Then, we can resample the momentum from the isotropic Gaussian $\mathcal{N}\left(0, \frac{m}{N}I_d \right)$, and the momentum update in the RB-SHMC is given by
\begin{gather}
p^* \leftarrow p^* - \Delta t_{\ell} \left[-\frac{1}{N}\nabla \log p_{\mathrm{prior}}(\theta^*) + \frac{1}{s}\sum_{i \in \xi_{\ell}} -\nabla \log p(y_i; \theta^*) + \frac{1}{N}\nabla G(\theta^*)\right],
\end{gather}
where $\xi_{\ell} = \{\xi_{\ell,1}, \ldots, \xi_{\ell, s}\} \subset \{1, \ldots, N\}$ denotes the random batch of data chosen in the ${\ell}^{th}$ step of the current iteration to approximate the likelihood of the entire dataset. The acceptance rate is simply $\min(1, \exp( \Delta G))$.

The time rescaling here is again equivalent to choosing $\beta=N$ , $U_1=-\frac{1}{N}\log p_{\mathrm{prior}}(\theta) + \frac{1}{N}\sum_{i = 1}^N -\log p(y_i; \theta)+\frac{1}{N}G$ and $U_2=-\frac{1}{N}G$ in Algorithm \ref{alg:SHMC}. 
The goal of time rescaling or choosing $\beta = N$ is to scale the summation term by $1/N$ so that the variance of $\chi^{\ell} := \frac{1}{s}\sum_{i \in \xi_{\ell}} -\nabla \log p(y_i; \theta^*)  - \frac{1}{N}\sum_{i = 1}^N -\nabla \log p(y_i; \theta^*)$ is controlled independent of $N$. Indeed, choosing $\beta = N$ is only for convenience in the analysis and incurs no essential change in the physical interpretation \cite{jin2020randomsecondorder}. In practice, one may perform the updates directly by 
\[
p^* \leftarrow p^* - \Delta t_{\ell} \left[-\nabla \log p_{\mathrm{prior}}(\theta^*) + \frac{N}{s}\sum_{i \in \xi_{\ell}} -\nabla \log p(y_i; \theta^*) + \nabla G(\theta^*)\right].
\]

\begin{remark}
\label{rem:meta}
The idea of adding computational sand is inspired by metadynamics \cite{laio2002escaping}, which is informally described as ``filling the free energy wells with computational sand''.
In some special cases, the modes of the posterior, and thus the locations of the wells in $U$, can be roughly estimated by the modes of the marginal distributions. 
Then an amenable $G$ can be designed in advance. 
Such an example is presented in Section \ref{subsubsec:gMM} below. 
Generally, however, the modes of the marginals fail to imply the location of the modes of the joint distribution. Hence, one may dynamically add a sequence of standard Gaussian kernels at the positions already visited by the proposals of the Hamiltonian simulation. 
In this case, Metropolis rejection steps are skipped in the first few trial runs. 
Then one collects all the added Gaussian kernels and obtains $U_1$, which is $U$ plus the sum of these Gaussian kernels. 
Interested readers may refer to \cite{laio2002escaping} for more details. We leave the theoretical justification of this proposal to future work.
\end{remark}

\section{Theoretical properties of SHMC and RB-SHMC}\label{sec:Theory&discussion}

In this section, we provide some theoretical justifications for the SHMC and RB-SHMC algorithms. First, in Section \ref{subsec:L}, we give some informal justification on why SHMC could potentially be better than splitting Monte Carlo methods based on overdamped Langevin dynamics. Second, in Section \ref{subsec:error}, we provide error estimates of the values of the approximate Hamiltonian generated from RB-SHMC under some mild regularity conditions. These results give practitioners some guidance on when to expect SHMC or RB-SHMC to succeed in applications.

\subsection{Benefits of Hamiltonian dynamics in sampling compared to the overdamped Langevin equations}\label{subsec:L}

As is well known, the Langevin sampling \cite{besag1994comments, roberts1996exponential, dalalyan2017theoretical, raginsky2017non} is another popular sampling scheme for the Gibbs distribution. 
Consider the overdamped Langevin equation:
\begin{gather}\label{eq:overdamped}
	dX = -\nabla U(X)\, dt + \sqrt{\frac{2}{\beta}}\, dW,
\end{gather}
where $\beta > 0$ is a known constant and $W$ denotes the $d$-dimensional Wiener process. 
It converges exponentially to its invariant measure, the Gibbs distribution \eqref{eq:Gibbs-U}, under mild regularity conditions \cite{markowich2000trend}. The Langevin sampler discretizes the overdamped Langevin equation \eqref{eq:overdamped}, by the Euler-Maruyama scheme
\[
 X(t^{n+1}) = X(t^n) - \Delta t_n\nabla U(X(t^n)) + \sqrt{\frac{2\Delta t_n}{\beta}}z_n,
 ~~z_n \sim \mathcal{N}(0, I_d).
\]
Hence the Markov chain generated by the Langevin sampler also exhibits an exponentially fast convergence rate up to the discretization error of order $\cO(\Delta t)$ \cite{dalalyan2019user}. 

In this section, we point out several possible benefits of the HMC-based samplers over the overdamped Langevin equation-based samplers like the RBMC and SGLD algorithms. In fact, the RBMC method uses the overdamped Langevin equation for generating the proposals in \textbf{Step 1} of Algorithm \ref{alg:generalSMC}, as the overdamped Langevin equation satisfies the detailed balance condition w.r.t. this invariant measure determined by $U_1$.
This may provide some heuristic justification for why in certain cases RB-SHMC may be preferred compared to RBMC, such as the examples described in sections \ref{subsec:test examples} and \ref{subsec:DBm} below.

In the discussion here, we assume that the potential $U$ is smooth and for the sake of simplicity, there is no splitting so $U_1 \equiv U$. 
Recall that we perform $L_n$ leapfrog steps with time step size $\Delta t_n$ in the $n^{th}$ iteration, which means that we evolve a Hamiltonian system w.r.t. $(X_{i^{(n)}}, P_{i^{(n)}})$ for $t \in \left[T_{i^{(n)}}^{(n-1)},T_{i^{(n)}}^{(n)}\right)$. 
Assume the particle index updated in the $n^{th}$ iteration is $i \equiv i^{(n)}$, where we omit the superscript for convenience. Then we approximate the momentum $P_i(t')$ at time $t' \in(T_{i}^{(n-1)},T_{i}^{(n)})$ by
\begin{gather*}
P_i(t') = P_i(T_{i}^{(n-1)}) + \int_{T_{i}^{(n-1)}}^{t'} -\nabla U(X_i(\tau)) d\tau.
\end{gather*}
 and thus
 \begin{gather}\label{eq:HMCequiv}
 \begin{split}
 X_i(T_{i}^{(n)}) & = X_i(T_{i}^{(n-1)}) + \int_{T_{i}^{(n-1)}}^{T_{i}^{(n)}} m^{-1} P_i(t') dt' \\
   	& = X_i(T_{i}^{(n-1)}) + \frac{L_n\Delta t_n}{m}P_i(T_{i}^{(n-1)}) - \frac{1}{m}
	\int_{T_{i}^{(n-1)}}^{T_{i}^{(n)}}\int_{T_i^{(n-1)}}^{t'}\nabla U(X_i(\tau)) d\tau dt' \\
   							& = X_i(T_{i}^{(n-1)}) - \widehat{\Delta t}_n\widetilde{\nabla U} + \sqrt{\frac{2\widehat{\Delta t}_n}{\beta}}z_n,
 \end{split}
 \end{gather}
with 
\[
\widetilde{\nabla U} =\frac{2}{(L\Delta t_n)^2}\int_{T_{i}^{(n-1)}}^{T_{i}^{(n)}}\int_{T_i^{(n-1)}}^{t'}\nabla U(X_i(\tau)) d\tau dt'
\]
and 
$\widehat{\Delta t}_n = \frac{(L_n\Delta t_n)^2}{2 m}$ and $z_n\sim\cN(0,I_d)$. 

Roughly, the numerical evolution in the $n^{th}$ iteration of SHMC is nearly equivalent to the simulation of an overdamped Langevin equation for time $\widehat{\Delta t}_n$. Hence, the effective time step of the overdamped Langevin equation is longer than the evolution time $L_n\Delta t_n$ of the corresponding Hamiltonian system, provided $L_n\Delta t_n > 2 m$. This means that running the Hamiltonian system with $L_n\Delta t_n$ large has a longer effective dynamics so it might approximate the equilibrium faster than the overdamped Langevin dynamics. This property may be the reason why SHMC is preferred in some applications over RBMC (see examples \ref{subsec:test examples}-\ref{subsec:DBm} below). 
Moreover, if the Hamiltonian step in \eqref{eq:HMCequiv} is accurately solved, it then exactly satisfies the detailed balance 
w.r.t. the target distribution. Discretizing the Hamiltonian using leapfrog might yield $\cO(\Delta t^2)$ error, which is better compared to the Euler-Maruyama discretization of overdamped Langevin dynamics with error $\cO(\Delta t)$ if we use the same step size.

Another benefit of HMC is that it can take advantage of the symplectic integrators to conserve the Hamiltonian even when the system evolves for a long time. Usually, with fixed time step $\Delta t$, the discretization error would increase exponentially with the evolution time.  However, the symplectic integration can be carried out for $e^{\cO\left(1/\Delta t\right)}$ time steps such that the numerical trajectories would always oscillate around the exact Hamiltonian trajectories within this time interval \cite{benettin1994hamiltonian, hairer2006geometric}. This guarantees that the Hamiltonian can be conserved within a reasonable tolerance for a long time, which is highly desired for HMC-based algorithms as the invariant measure is related to the Hamiltonian directly.
Moreover, due to the time-reversibility and volume-preservation property of the leapfrog integrator, HMC exhibits better scalability compared to RWM and MALA in some cases. For instance, when the potential of the target distribution is a sum of i.i.d. terms such as \eqref{eq:manybodypotential} and satisfies mild regularity conditions, HMC can allow a larger time step $\Delta t = \cO(N^{-1/4})$, and thus requires less steps to traverse the state space, as the number $N$ of i.i.d. terms tends to infinity \cite{neal2011mcmc,beskos2013optimal}.

\subsection{Error estimates of the random batch method}\label{subsec:error}

In this section, we derive an error bound for the values of the Hamiltonian approximated by the random batch strategy in one iteration of RB-SHMC (the process between any two consecutive resampling steps of the momentum). We analyze how the random batch strategy affects the value of the Hamiltonian corresponding to $U_1$. 
Note that the error of the value of the Hamiltonian determines the deviation of the invariant measure so this estimate can provide us some insight on the impact of the random batch approximation on the invariant measures.

Without loss of generality, we denote the beginning of an iteration as $t = 0$, and consider the following Hamiltonian system 
\begin{gather}\label{eq:Hamilton_psi}
\begin{split}
\dot{\bm{x}}&=\frac{\bm{p}}{m},\\
\dot{\bm{p}}&=-\nabla V_1(\bm{x})+\frac{1}{N_J}\sum_{j \in \cJ} \nabla \psi_j(\bm{x})
\end{split}
\end{gather}
and its counterpart after applying random batch 
\begin{gather}\label{eq:Hamilton_psi_rb}
\begin{split}
\dot{\tbx}&=\frac{\tbp}{m},\\
\dot{\tbp}&=-\nabla V_1(\tbx)+\frac{1}{s}\sum_{j \in \xi} \nabla \psi_j(\tbx). 
\end{split}
\end{gather}
Here $N_J$ is the size of the index set $\cJ$ and $\xi$ denotes a random subset of size $s$. 
We remark that introducing the notations $V_1$ and $\psi_j$ enables us to unify the Hamiltonian systems of the two examples discussed in Section \ref{subsec:RB-SHMC}. In particular, in the examples of interacting particle systems (Sections \ref{subsec:test examples}, \ref{subsec:DBm} and \ref{subsec:lj}), $V_1(\bm{x}) = V(\bm{x}) / (w (N-1))$ is a scaling of the confining potential of the current particle $i$ which is located at $\bm{x}$ and $\psi_j(\bm{x}) = \psi(\bm{x}; q_j) = -\phi_1(\bm{x}-q_j)$ is the interaction potential between particle $i$ and another particle $j \neq i$ located at $q_j$. When it comes to the example of the Bayesian posterior inference (Section \ref{subsubsec:gMM}), $V_1(\cdot) = (-\log p_{\mathrm{prior}}(\cdot) + G)/N$ is a scaling of the sum of the $\log$ prior density and the auxiliary Gaussian kernel added to flatten the landscape of the potential and $\psi_j$ corresponds to the $\log$-likelihood term $-\log p(y_j; \cdot)$.

We first clarify the notations used in the following analysis. 
Fix the number of leapfrog steps and the time step in each iteration to be $L$ and $\Delta t$ and denote $T := L\Delta t$. 
The initial values (i.e., the values at the beginning of one iteration) are denoted by $X^0$ and $P^{0}$.
Denote the random batch selected in the $\ell^{th}$ leapfrog step as $\xi_{\ell}, \ell = 1, \ldots, L$. 
The Kolmogorov extension theorem \cite{durrett2019probability} guarantees the existence of a probability space $(\Omega, \cF, \bbP)$ such that $\{X^0, P^{0}, \xi_{\ell}, \ell = 1, \ldots, L\}$ are independent random variables on this space. 
The $L^2(\Omega, \bbP)$ norm is denoted as $\|\cdot\| = \sqrt{\bbE|\cdot|^{2}}$, where $\bbE$ is the expectation with respect to the probability measure $\bbP$.

Now we start to estimate the difference of $H(\tbx(T), \tbp(T))$ and $H(\bm x(T), \bm p(T))$, where $(\bm x(\cdot), \bm p(\cdot))$ is the solution of the fully coupled Hamiltonian system \eqref{eq:Hamilton_psi} and therefore $H(\bm x(T), \bm p(T)) = H(\bm x(0), \bm p(0))$. 

For error bound estimates, we assume that the Hamiltonian is sufficiently smooth. 
\begin{assumption}\label{aspt:smoothness}
$V_1\in C^1$ and $\psi\in C^2$ with $\nabla V_1, \nabla \psi$ and $\nabla^2 \psi$ being bounded.
\end{assumption}

\begin{remark}
In cases where the actual confining potential has an unbounded gradient, one can still split $V$ into a bounded and an unbounded part so the unbounded part can be used in the Metropolis rejection step. However, this may not be needed in practice as the particles usually move in a bounded domain so the unboundedness of $V$ is not very relevant.
\end{remark}

\medskip
 The main results are as follows.
\begin{theorem}\label{thm:E DeltaH}
Under Assumption \ref{aspt:smoothness} and assume $\Delta t$ is sufficiently small such that $T(1+T^2)\Delta t \ll 1$. If
$\bm x(0) = \tbx(0)=X^0$, and $\bm p(0) = \tbp(0)=P^0$, and the fourth moment of the initial momentum $\mathbb{E}[|P^0|^4]$ is bounded, then the error from using random batches can be bounded above as follows:
\begin{gather}\label{eq:strongDeltaH}
 \left\|H(\tbx(T), \tbp(T)) - H(\bm x(T), \bm p(T))\right\| \leq C\sqrt{T(1+T^2)\Delta t}, 
 \end{gather}
and for any test function $\varphi \in \mathcal{C}_b^{\infty}$, 
\begin{gather}\label{eq:weakDeltaH}
\left|\mathbb{E}\left[\varphi\left(H(\tbx(T), \tbp(T))\right) - \varphi\left(H(\bm x(T), \bm p(T))\right)\right]\right| \leq C_{\varphi}T(1+T^2)\Delta t, 
\end{gather}
 where $C, C_{\varphi}>0$ are constants independent of $T$ and $N$ and $C_{\varphi}$ depends on $\varphi$. 
\end{theorem}

\begin{remark}
As mentioned in Remark \ref{rmk:error2}, the estimation of the error introduced by random batch guarantees that the systematic error of our method vanishes as $\Delta t_n \rightarrow 0$.
\end{remark}

We also remark that the bounded fourth moment conditions on the momentum $P^{0}$ are trivially satisfied as it usually sampled from a multivariate Gaussian distribution. In the proof, we need to define the filtration $\{\mathcal{F}_{\ell}\}_{1\leq \ell\leq L}$ by
\begin{gather}
\mathcal{F}_{\ell} := \sigma\left( X^0, P^0, \{\xi_k\}_{1\leq k \leq \ell} \right), 
\end{gather}
the $\sigma$-algebra generated by the initial values $X^0, P^0$ and the first $\ell$ random batches $\xi_k, k = 1, \ldots, \ell$. 
Obviously, $\mathcal{F}_{\ell+1} = \sigma\left(\mathcal{F}_{\ell} \cup \sigma\left(\xi_{\ell+1}\right)\right)$.
Denote the error of the random approximation of the summation term by:
\begin{gather}
\chi^{\ell}(t) := \frac{1}{s} \sum_{j\in\xi_{\ell}}\nabla \psi_j(\tbx(t)) - \frac{1}{N_J} \sum_{j \in \cJ}\nabla \psi_j(\tbx(t)).
\end{gather}
Clearly, $\chi^{\ell}$ is $\mathcal{F}_{\ell}$-measurable and for any $\ell\geq 0, q > 0$,
\begin{gather}
\sup_{t\in[t_{\ell}, t_{\ell+1})}\mathbb{E}\left[\left|\chi^{\ell}(t)\right|^q\middle|\mathcal{F}_{\ell}\right] \leq 2^q \|\nabla \psi\|_{\infty}^q. 
\end{gather}

With these, the following upper bound, which will be used in the proof of Theorem \ref{thm:E DeltaH}, can be easily obtained: 
\begin{lemma}\label{lem:E P^q}
Under conditions in Theorem \ref{thm:E DeltaH}, it holds that, 
\begin{gather}\label{eq:E P^q-global}
\sup_{0 \leq t\leq T}\mathbb{E}\left[ \left|\tbp(t)\right|^4 \right]\leq C(1 + T^4), 
\end{gather}
where $C>0$ is a constant independent of $N, s, \xi_{k}, k = 1, \ldots, L$.
\end{lemma}
The proof of this simple lemma is deferred to Appendix \ref{apdx:pf lem E P^q}. We are now able to provide the proof of the error estimates for the Hamiltonian.
\begin{proof}[Proof of Theorem \ref{thm:E DeltaH}]

Define
\begin{gather}
\Delta H(t) := H(\tbx(t), \tbp(t)) - H(\bm x(t), \bm p(t)) = H(\tbx(t), \tbp(t)) - H^0,
\end{gather}
where $H^0 = H(X^0, P^0)$ is the initial value of the Hamiltonian since $H$ is conserved under the evolution \eqref{eq:Hamilton_psi}. Consequently, it is straightforward to find
\begin{gather}
 \frac{d}{dt}\Delta H(t) = \frac{\tbp(t)}{m}\cdot \chi^{L_t}(t), 
\end{gather}
where 
\begin{gather}
L_t = \lfloor \frac{t}{\Delta t}\rfloor.
\end{gather}

Hence, 
\begin{gather}\label{eq:d/dt EDeltaH^2}
\begin{split}
&\frac{d}{dt} \|H(\tbx(t), \tbp(t)) - H(\bm x(t), \bm p(t))\|^2 
= 2\mathbb{E} \left[\Delta H(t)\frac{\tbp(t)}{m}\cdot \chi^{L_t}(t)\right] \\
&=  2\mathbb{E} \left[\left(\Delta H(t_{L_t}) + \frac{\tbp(\hat{t})}{m}\cdot \chi^{L_t}(\hat{t})(t - t_{L_t})\right)\frac{\tbp(t)}{m}\cdot \chi^{L_t}(t)\right] \\
&=:  I_1 + I_2,
\end{split}
\end{gather}
where $\hat{t} \in (t_{L_t}, t)$ by the mean value theorem for one variable functions. 

\textbf{Step 1} --- Estimation of $I_1$:

Note that $\Delta H(t_{L_t}) \in \mathcal{F}_{L_t-1}$ and $\tbp(t_{L_t}) \in \mathcal{F}_{L_t-1}$. Then one has
\begin{gather*}\mathbb{E}\left[\Delta H(t_{L_t}) \frac{\tbp(t_{L_t})}{m}\chi^{L_t}(t_{L_t})\right] 
=  \mathbb{E}\left[\Delta H(t_{L_t}) \frac{\tbp(t_{L_t})}{m}\cdot \mathbb{E}\left(\chi^{L_t}(t_{L_t})\middle|\mathcal{F}_{L_t-1}\right)\right]  = 0.
\end{gather*}
Therefore,
\begin{gather*}
\begin{split}
& \ \mathbb{E}\left[\Delta H(t_{L_t}) \frac{\tbp(t)}{m}\chi^{L_t}(t)\right]
= \mathbb{E}\left[\Delta H(t_{L_t}) \left(\frac{\tbp(t)}{m}\chi^{L_t}(t) - \frac{\tbp(t_{L_t})}{m}\chi^{L_t}(t_{L_t})\right)\right] \\
= & \ \mathbb{E}\left[\Delta H(t_{L_t}) \mathbb{E}\left(\frac{\tbp(t)}{m}\chi^{L_t}(t) - \frac{\tbp(t_{L_t})}{m}\chi^{L_t}(t_{L_t})\middle|\mathcal{F}_{L_t}\right)\right] \\
= & \ \mathbb{E}\left[\Delta H(t_{L_t}) \int_{t_{L_t}}^t \mathbb{E}\left(\frac{\dot{\tbp}(t')}{m}\chi^{L_t}(t') 
+ \frac{\tbp(t')}{m}\dot{\chi}^{L_t}(t')\middle|\mathcal{F}_{L_t}\right)\, dt'\right] \\
\leq & \ C(1+T^2)\Delta t\|\Delta H(t_{L_t})\|
\leq C(1+T^2)\Delta t \left(\|\Delta H(t)\| + \left\|\frac{\tbp(\hat{t})}{m}\chi^{L_t}(\hat{t})(t - t_{L_t})\right\| \right)\\
\leq & \ C(1+T^2)\Delta t\left(\|\Delta H(t)\| + C(1+T)\Delta t\right)
\leq C(1+T^2)\Delta t \|\Delta H(t)\|.
\end{split}
\end{gather*}

\textbf{Step 2} --- Estimation of $I_2$:
\begin{gather*}
\begin{split}
& \ \mathbb{E}\left[\frac{\tbp(\hat{t})}{m}\chi^{L_t}(\hat{t})\frac{\tbp(t)}{m}\chi^{L_t}(t)\right] \\
\leq & \ \left\{\mathbb{E}\left[\left|\frac{\tbp(\hat{t})}{m}\right|^4\right]\mathbb{E}\left[\left|\chi^{L_t}(\hat{t})\right|^4\right]\mathbb{E}\left[\left|\frac{\tbp(t)}{m}\right|^4\right]\mathbb{E}\left[\left|\chi^{L_t}(t)\right|^4\right]\right\}^{\frac{1}{4}} \leq C(1+T^2).\\
\end{split}
\end{gather*}
Therefore, $I_2$ is controlled by $C(1+T^2)\Delta t$. 

Combining the above two steps gives
\begin{gather}
\frac{d}{dt} \|\Delta H(t)\|^2 \leq C(1+T^2)\Delta t\left(\|\Delta H(t)\| + 1\right).
\end{gather}
Recall that $\Delta t$ is assumed to be sufficiently small. Then one has
\begin{gather}
\|H(\tbx(T), \tbp(T)) - H(\bm x(T), \bm p(T))\| \leq C\sqrt{T(1+T^2)\Delta t}. 
\end{gather}

Similarly,  \eqref{eq:weakDeltaH} follows from
\begin{gather*}
\begin{split}
& \ \left|\mathbb{E}\left[\varphi\left(H(\tbx(T), \tbp(T))\right) - \varphi\left(H(\bm x(T), \bm p(T))\right)\right]\right| \\
 = & \ \left|\int_0^T\mathbb{E}\left[\varphi'\left(H(\tbx(t), \tbp(t))\right)\frac{\tbp(t)}{m}\chi^{L_t}(t)\right]dt\right|\\
 = & \ \left|\int_0^T\mathbb{E}\left[\varphi'\left(H(\tbx(t), \tbp(t))\right)\frac{\tbp(t)}{m}\chi^{L_t}(t) - \varphi'\left(H(\tbx(t_{L_t}), \tbp(t_{L_t}))\right)\frac{\tbp(t_{L_t})}{m}\chi^{L_t}(t_{L_t})\right]dt\right| \\
\leq & \ \left|\int_0^T\int_{t_{L_t}}^t\mathbb{E}\left[\varphi''\left(H(\tbx(t'), \tbp(t'))\right)\left(\frac{\tbp(t')}{m}\chi^{L_t}(t')\right)^2\right]dt'dt\right| \\
& + \left|\int_0^T\int_{t_{L_t}}^t\mathbb{E}\left[\varphi'\left(H(\tbx(t'), \tbp(t'))\right)\left(\frac{\dot{\tbp}(t')}{m}\chi^{L_t}(t') + \frac{\tbp(t')}{m}\dot{\chi}^{L_t}(t')\right)\right]dt'dt\right| \\
\leq & \ (\|\varphi'\|_{\infty} + \|\varphi''\|_{\infty})CT(1+T^2)\Delta t. 
\end{split}
\end{gather*}

\end{proof}

\section{Numerical examples}\label{sec:Numerical examples}

In this section, we conduct some numerical experiments to demonstrate the computational efficiency of SHMC and RB-SHMC.
First, we give an artificial example with regular interactions to test the influence of the length of Hamiltonian dynamics in an iteration. 
In the second example, we simulate the Dyson Brownian motion for $N\gg 1$. 
The advantage of potential splitting and the high efficiency of random batch will manifest itself in this example, where the interaction is singular and of long range.
Third, we consider a distribution with double well potential to demonstrate the benefit of potential splitting even in regular interaction cases, and then consider a classical example in Bayesian inference: estimating the locations of a two-dimensional Gaussian mixture model using SHMC and RB-SHMC. Finally, we consider the Gibbs distribution of a three-dimensional interacting particle system, the interaction potential of which is modelled by the Lennard-Jones potential. This example, which is of higher dimension than previous examples, further demonstrates the benefit of RB-SHMC in practice.

\subsection{A test example}\label{subsec:test examples}

In the following toy example, we explore how $L_n$, the number of leapfrog steps in the $n^{th}$ iteration, can influence the efficiency of the RB-SHMC algorithm. 
For the purpose of a relatively fair comparison between algorithms with different evolution time per iteration, we define the evolution time $T_E$ up to the $M^{th}$ iteration of Algorithm \ref{alg:SHMC} by
\begin{gather}\label{eq:EvoTime}
	T_E(M) = \frac{1}{N}\sum_{n = 1}^M L_n\times \Delta t_n.
\end{gather}
The rescaling factor $1/N$ arises from the fact that Algorithm \ref{alg:SHMC} updates only one ``particle'' per iteration and thus on average $N$ iterations are needed to update all the variables. 

Consider the simple one-dimensional Langevin equation
\begin{equation}\label{eq:test example}
	\begin{array}{l}
		d X^{i}=P^{i} d t, \\
		d P^{i}=\left(-\alpha X^{i}+\frac{1}{N-1} \sum_{j: j \neq i} \frac{X^{i}-X^{j}}{1+\left|X^{i}-X^{j}\right|^{2}}-\gamma P^{i}\right)d t+\sqrt{\frac{2\gamma}{\beta}} d W^{i}.
	\end{array}
\end{equation}
Clearly, the interaction is smooth and bounded and its derivative is also bounded. 
The corresponding Gibbs distribution is
\begin{gather}
\mu \propto \exp\left[-\frac{\beta}{2}\left(\alpha \sum_i X_i^2 - \frac{1}{N-1}\sum_{i, j:i<j}\ln (1+|X_i-X_j|^2) + \sum_i P_i^2\right)\right].
\end{gather}

Consider a system of $N=500$ particles sampled from the uniform distribution on $[-10, 10]$. 
Since the interaction kernel is regular, we move all the particles simultaneously using $\phi_1 := \phi$ and omit the Metropolis rejection step since $\phi_2 = 0$. 
Choosing $\alpha = 1, \, \beta = 1$ and fixing the batch size $s =1$, the $\ell^{th}$ leapfrog step includes the following updates
\begin{gather*}
 	\begin{split}
 		& \bm p \leftarrow \bm p - \frac{\Delta t_n}{2} \left(\alpha \bm x - \frac{\bm x - \bm \eta_{\ell}}{1+|\bm x - \bm \eta_{\ell}|^2}\right); \\
 		& \bm x \leftarrow \bm x + \Delta t_n \bm p; \\
 		& \bm p \leftarrow \bm p - \frac{\Delta t_n}{2} \left(\alpha \bm x - \frac{\bm x - \bm \eta_{\ell}}{1+|\bm x - \bm \eta_{\ell}|^2}\right),
 	\end{split}
 \end{gather*}
 where the $i^{th}$ element of $\bm \eta_{\ell}$ is $\eta_{\ell,i} = x_{\xi_i}$, with $\xi_i \in \{1, 2, \ldots, N\}\setminus\{i\}$ for $i = 1, 2, \ldots, N$. 
 Fixing the time step $\Delta t_n \equiv 0.02$, we run RB-SHMC with $L_n \equiv 100$ and $L_n \equiv 10$ respectively for $T_E = 100$ units of evolution time to see how the efficiency can be influenced by different values of $L_n$. 

\begin{figure}[htb]
\centering
\begin{subfigure}[t]{0.32\textwidth}
\centering
\includegraphics[width=\textwidth]{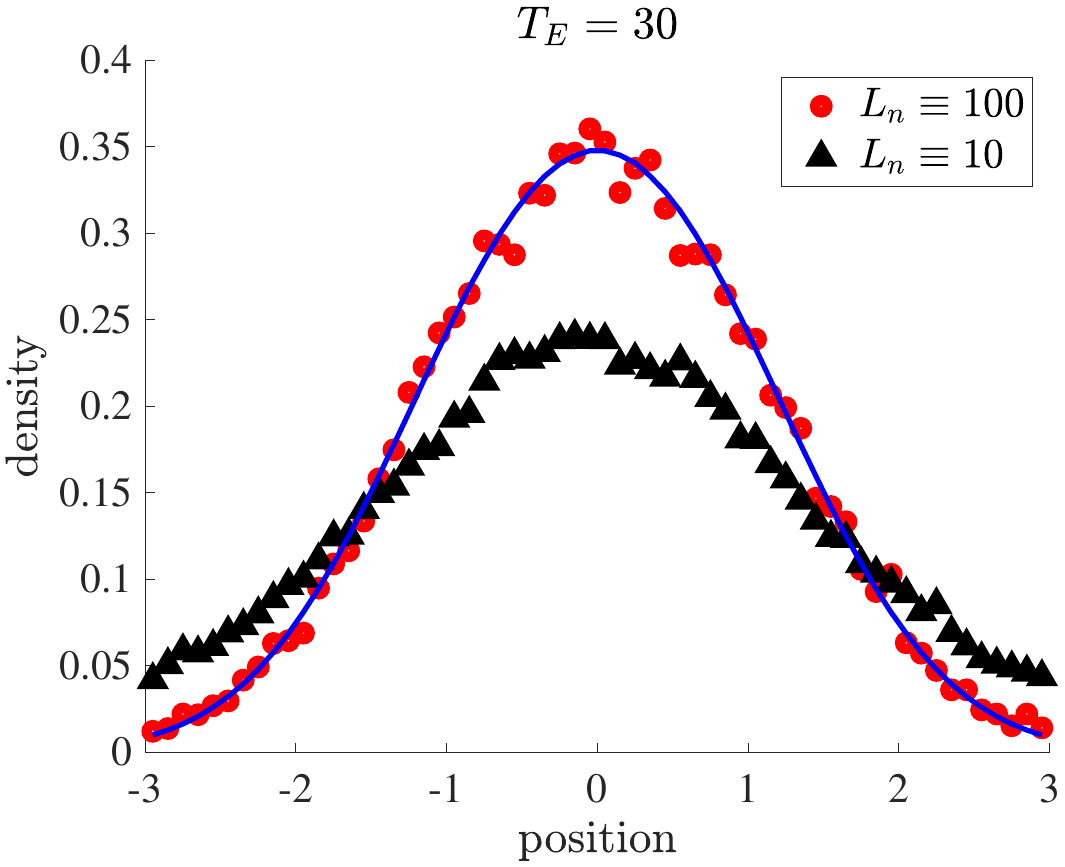}
\end{subfigure}
\hfill
\begin{subfigure}[t]{0.32\textwidth}
\centering
\includegraphics[width=\textwidth]{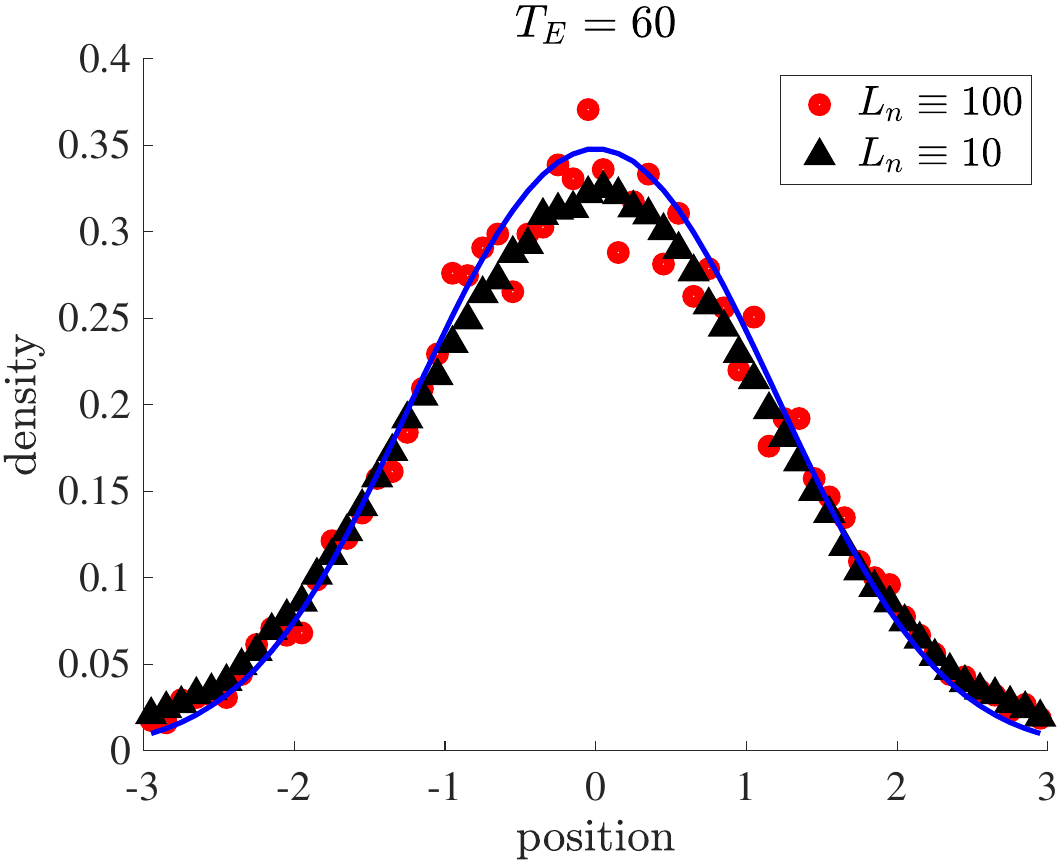}
\end{subfigure}
\hfill
\begin{subfigure}[t]{0.32\textwidth}
\centering
\includegraphics[width=\textwidth]{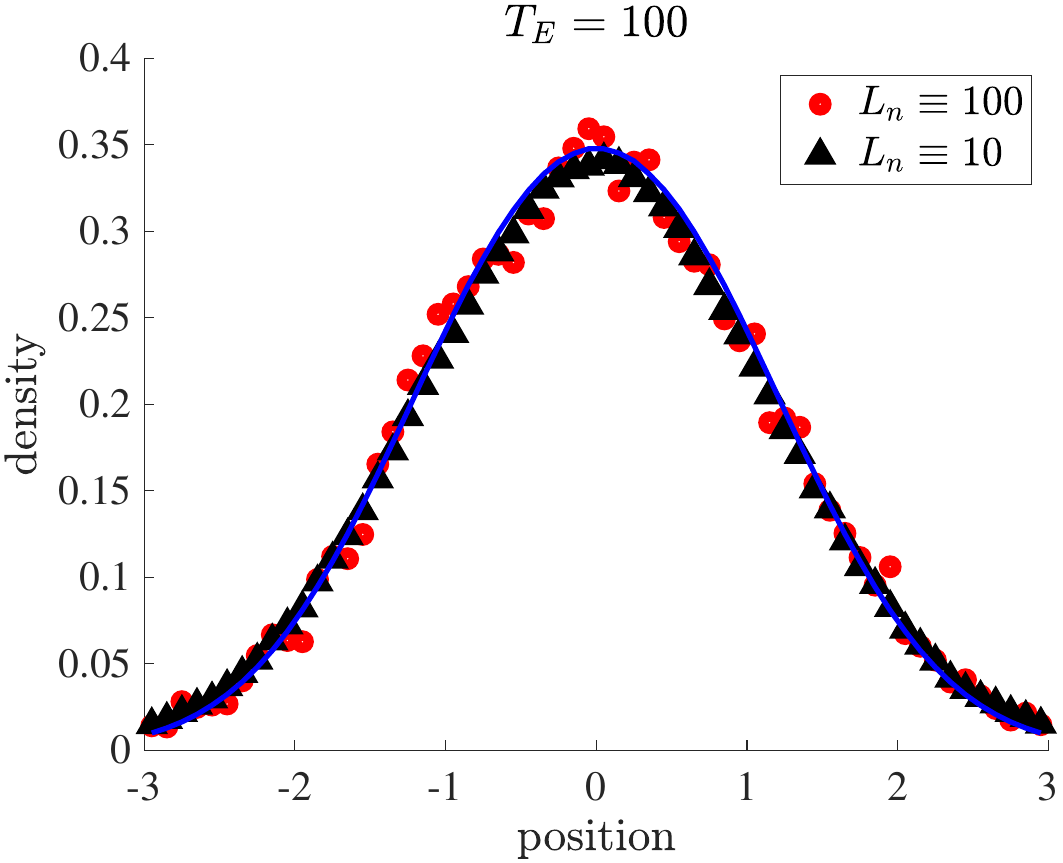}
\end{subfigure}
\caption{Empirical densities obtained by RB-SHMC with different $L_n$, numbers of leapfrog steps, at three different evolution time points $T_E = 30, 60, 100$.} The red dots are for $L_n=100$ leapfrog steps and the black triangles are for $L_n=10$ leapfrog steps in each iteration. The blue curve is the reference density obtained by running HMC, 25 leapfrog steps with time step size $\Delta t = 0.01$ per iteration, for $4e8$ iterations. The empirical density obtained by RB-SHMC with large $L_n$ (red dots), i.e. more leapfrog steps in each iteration, approximates the reference density better than the empirical density obtained by RB-SHMC with small $L_n$ (black triangles).
\label{fig:3 empirical densities}
\end{figure}

Figure \ref{fig:3 empirical densities} shows the empirical densities at ``evolution times'' $T_E = 30, \ 60, \ 100$ obtained by the simple ``bin-counting''. Specifically, the empirical density in the $j$-th bin is approximated by $\bar\mu_j \approx N_j/(N_{tot}h)$, where $N_j$ is the number of particles in the $j$-th bin during the entire sampling process and $N_{tot}$ is the number of total particles (clearly, $N_{tot}=N_sN$ where $N_s$ is the number of iterations in sampling). 
One can observe that RB-SHMC with larger $L_n$ approaches the equilibrium faster, meaning that it can have shorter burn-in phase. This observation is consistent with the heuristic explanation provided in Section \ref{subsec:L}.

To quantify the effects of $L_n$ in different stages of sampling, we use the following quantity (which we call ``relative error'') to gauge the 
error of the empirical density
\[
\tilde{U} := \sum_j \left|\frac{N_j}{N_{tot}} - \frac{\tilde{N}_j}{\tilde{N}_{tot}}\right|,
\]
where $j$ denotes the bin index, and the quantities with tildes are the reference quantities obtained by HMC (25 leapfrog steps with time step size $\Delta t = 0.01$ per sampling iterations, for a total of $4e8$ iterations).

\begin{figure}[htb]
\centering
\begin{subfigure}[t]{0.48\textwidth}
\centering
\includegraphics[width=\textwidth]{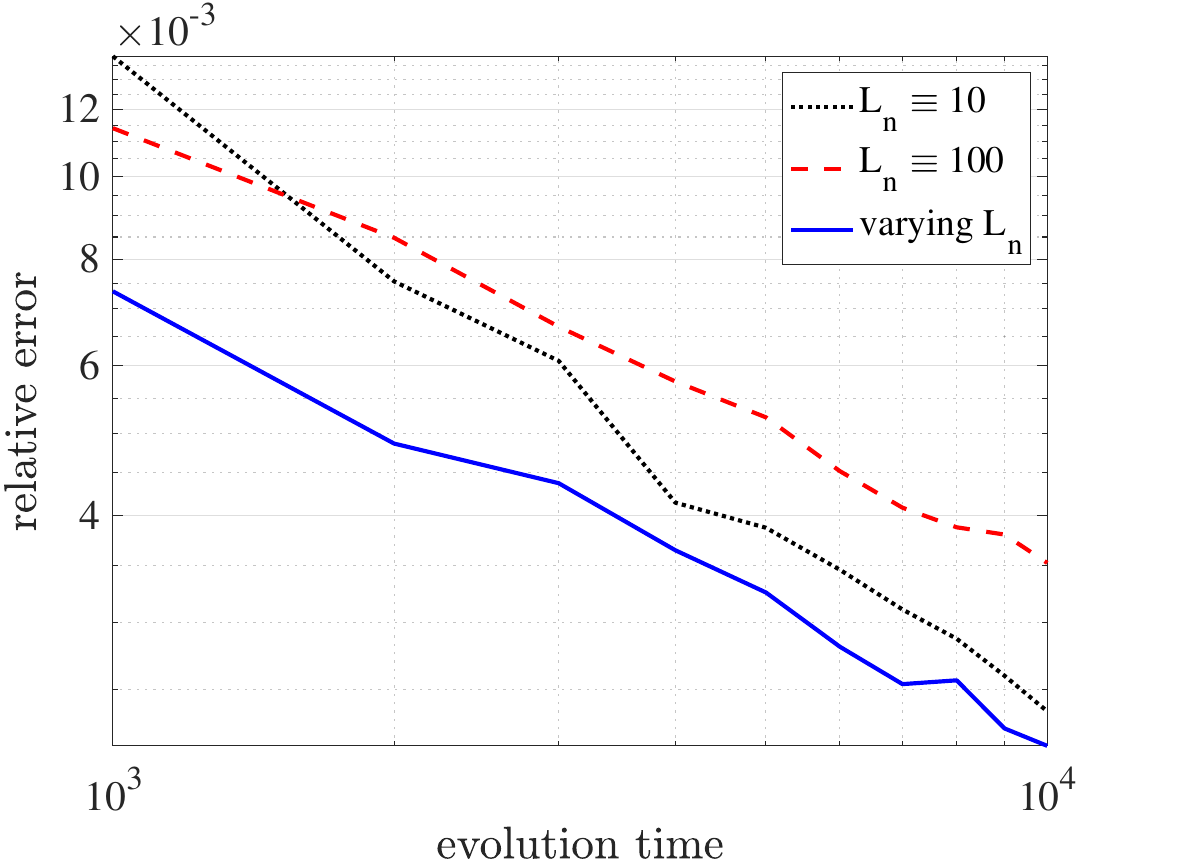}
\end{subfigure}
\hfill
\begin{subfigure}[t]{0.48\textwidth}
\centering
\includegraphics[width=\textwidth]{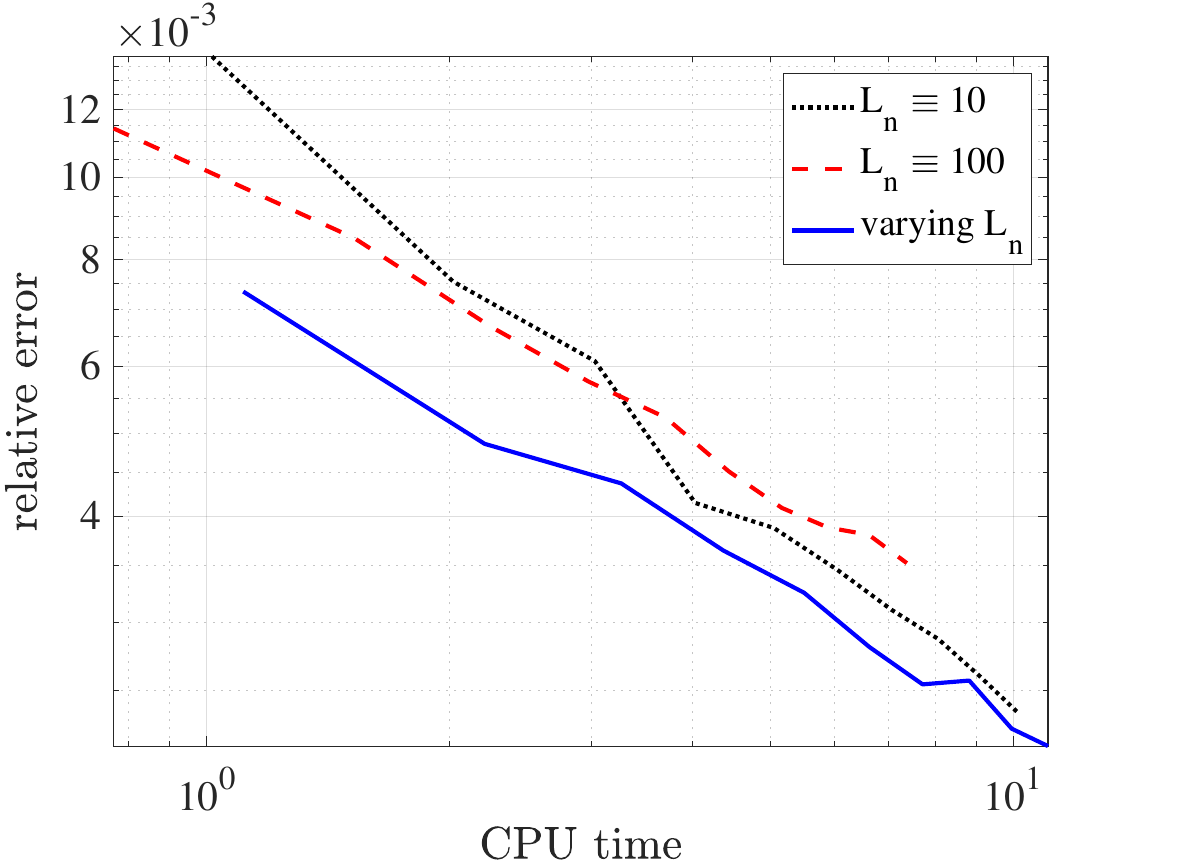}
\end{subfigure}
\caption{Left panel: relative error versus evolution time. Right panel: relative error versus CPU time. In both panels, blue curve is for RB-SHMC with varying $L_n$, red dashed curve is for RB-SHMC with $L_n \equiv 100$ while black dotted curve is for RB-SHMC with $L_n \equiv 10$. The relative error of the empirical density obtained by RB-SHMC with varying $L_n$ (blue curve) decays faster than those of RB-SHMC with fixed $L_n$. }
\label{fig:RBHMC with different amounts of random batches}
\end{figure}

As can be seen from Figure \ref{fig:RBHMC with different amounts of random batches}, the experiment with larger $L_n$ is more efficient when the evolution time is not very long as it has shorter burn-in phase. However, the systematic error becomes larger when one runs the sampling for longer time. Hence, an adaptive strategy can be used: one uses a large $L_n$ in the early sampling phase, and when the distribution is close to equilibrium one can then switch to smaller $L_n$ so that the results can be more accurate. Here, for comparison, we set $L_n=100$ for $T_E\le 100$ and then set $L_n=10$.
The error vs. the evolution time and CPU time are plotted in Figure \ref{fig:RBHMC with different amounts of random batches}.  One can see that running RB-SHMC with a large $L_n$ for a few units of evolution time at the beginning accelerates the empirical density to approach its equilibrium state while reducing $L_n$ to a smaller number after reaching a quasi-equilibrium yields a smaller error. The adaptive method can thus be efficient.

\subsection{Dyson Brownian motion}\label{subsec:DBm}

In this section, we demonstrate the efficiency of RB-SHMC by simulating the Dyson Brownian motion with singular interaction kernels, which models the eigenvalues of certain random matrices \cite{dyson1962brownian,erdos2017dynamical}: 
\begin{gather}\label{eq:dysonBMips}
d\lambda_i(t) = \left(-\lambda_i(t) + \frac{1}{N-1} \sum_{j: j\neq i} \frac{1}{\lambda_i(t) - \lambda_j(t)}\right) dt + \sqrt{\frac{2}{ N-1}}dW_i(t), 
\end{gather}
where $\{\lambda_i(t)\}$'s denote the eigenvalues and $\{W_i(t)\}$'s are independent standard Brownian motions. 
We have replaced $N$ in the original model with $N-1$, which is nonessential when $N\gg 1$.
It has been shown in \cite{tao2012topics} that in the $N\to\infty$ limit, these eigenvalues follow a distribution $\rho$ that satisfies
\begin{gather}\label{eq:dysonlimiteq}
\partial_t\rho(x, t)+\partial_x(\rho(u-x))=0, ~~u(x, t)=\pi(H\rho)(x, t)=\mathrm{p.v.}\int_{\mathbb{R}}\frac{\rho(y,t)}{x-y}\,dy,
\end{gather}
where $H(\cdot)$ is the Hilbert transform on $\mathbb{R}$ and p.v. is the standard notation for integrals evaluated using the Cauchy principal value. The mean field equation \eqref{eq:dysonlimiteq} has the following invariant measure:
\begin{gather}\label{eq:semicircle}
\rho(x)=\frac{1}{\pi}\sqrt{2-x^2},
\end{gather}
which is the celebrated Wigner semicircle law. 
Clearly, \eqref{eq:dysonBMips} has an invariant measure 
\[
\mu \propto \exp\left[-\Big(\frac{N-1}{2}\sum_ix_i^2-\sum_{i,j: i<j}\ln|x_i-x_j|\Big)\right], 
\]
where the interaction $\phi(x_i-x_j) = -\ln(|x_i-x_j|)$ is singular. 
For samples $(\lambda_1,\cdots, \lambda_N)\sim \mu$, we expect that the random empirical measure $\mu_N=\frac{1}{N}\sum_i \delta(x-\lambda_i)$ will be close to $\rho$ in the weak topology. Below, we collect $N_s$ such configurations $(\lambda_1,\cdots,\lambda_N)$ by sampling from $\mu$ using our sampling methods, and then compare
the empirical measure of these $N_s N$ samples to the target Wigner semicircle law $\rho$.

Following \cite{li2020random}, we use the surrogate potential $\phi_1(x_i-x_j) = \ln(100)-100|x_i-x_j|+1$ when $0<|x_i-x_j|<0.01$ to remove the singularity while performing RB-SHMC sampling. 

\begin{table}[h]
       \centering
	\begin{tabular}{|c|c|c|c|}
	\hline 
	\multicolumn{2}{|c|}{} & $L_n$  & $\Delta t_n$ \\
	\hline
	\multirow{3}{*}{RB-SHMC} & $n \leq 10^5$ & 100 & $2\times10^{-4}$\\ 
	\cline{2-4}
 						& $10^5 < n \leq 4\times10^5$ & 20 & $2\times10^{-4}$\\ 
 	\cline{2-4}
 						& $n > 4\times10^5$ & 10 & $10^{-4}$\\ 			
	 \hline 
	 \multicolumn{2}{|c|}{RBMC-fixed} & 10  & $10^{-4}$ \\
	\hline
	\multirow{3}{*}{RBMC-varying} & $n \leq 2\times10^5$ & 100 & \multirow{3}{*}{$10^{-4}$}\\ 
	\cline{2-3}
 						& $2\times10^5 < n \leq 8\times10^5$ & 20 & \\ 
 	\cline{2-3}
 						& $n > 8\times10^5$ & 10 & \\ 
	\hline 
	\end{tabular} 
\caption{The specific choices of the number of discretization steps $L_n$ and the stepsize $\Delta t_n$ in each iteration of RB-SHMC, RBMC-fixed and RBMC-varying in the experiment for Dyson Brownian motion. }
\label{tabl:DBm}
\end{table}

Recall that the example in Section \ref{subsec:test examples} indicates that running $L_n$ leapfrog steps with $L_{n}$ large in the burn-in phase can accelerate the convergence to a quasi equilibrium state while a small $L_n$ in later iterations can reduce the error.  Hence, we adopt a dynamic leapfrog steps in the simulation for this example. 
Specifically, we choose $L_n = 100$ in the first $10^5$ sampling iterations and $L_n = 20$ in the next $3 \times 10^5$ iterations. 
Then we reduce $L_n$ to 10 and keep it fixed until the end of the sampling process.  
The corresponding time step size $\Delta t_n$ also varies in different sampling phases (detailed in table \ref{tabl:DBm}). 
We compare it to the RBMC which is a splitting Monte Carlo based on overdamped Langevin dynamics. With $L_n$ denoting the number of discretization steps in the $n$-th iteration, we consider both fixed $\Delta t_n$ and $L_n$ parameters (denoted as ``RBMC-fixed'') and varying $L_n$ (denoted as ``RBMC-varying''). Table \ref{tabl:DBm} shows the specific choice of these parameters for the three methods considered in this example.

\begin{figure}[htb]
\centering
\begin{subfigure}[t]{0.48\textwidth}
\centering
\includegraphics[width=\textwidth]{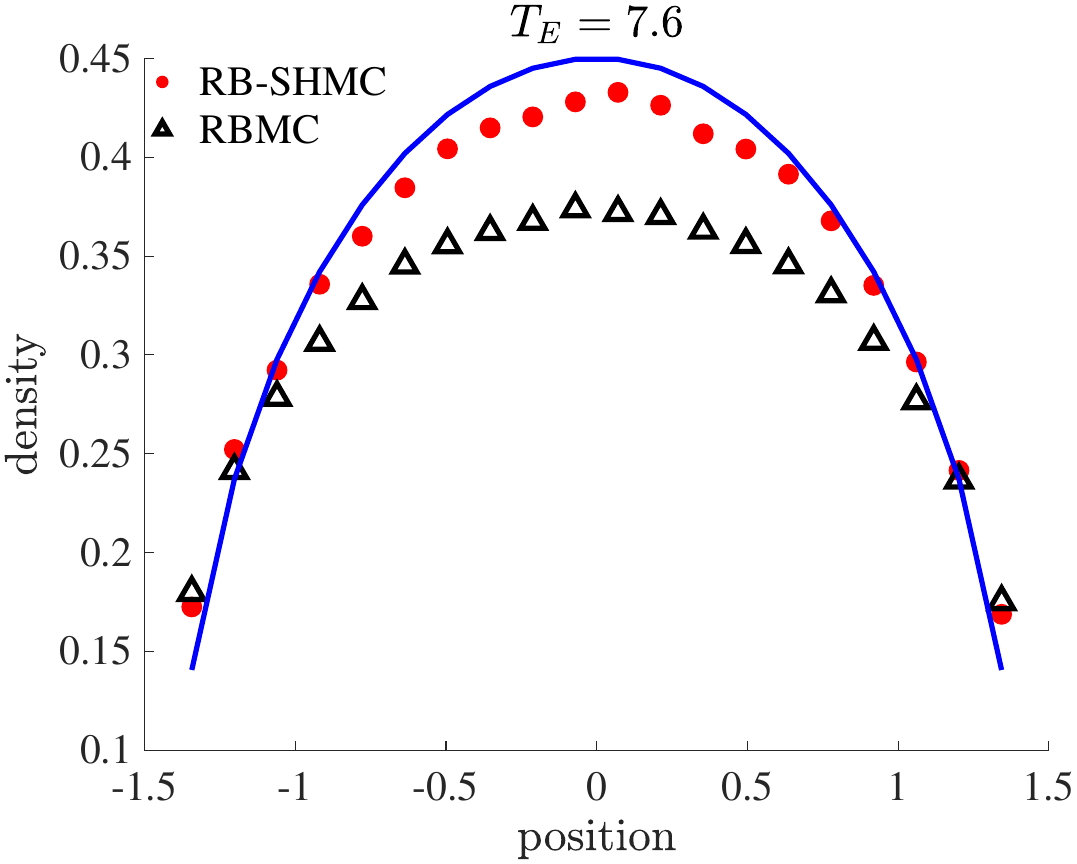}
\end{subfigure}
\hfill
\begin{subfigure}[t]{0.48\textwidth}
\centering
\includegraphics[width=\textwidth]{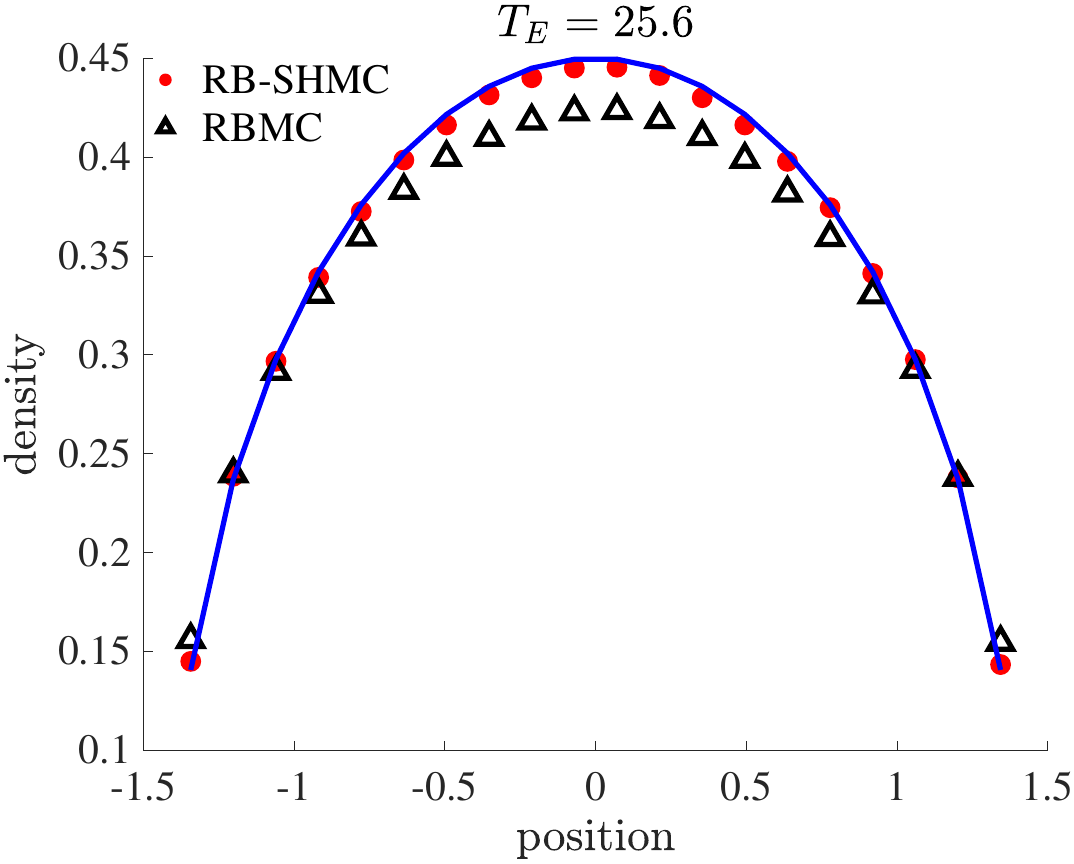}
\end{subfigure}
\caption{The empirical densities obtained by RB-SHMC and RBMC at different evolution times. Specifically, the left panel is for evolution time $T_E = 7.6$, or equivalently $10^6$ iterations, and the right panel is for $T_E = 25.6$, or equivalently $10^7$ iterations.} The red dots and the black triangles are the empirical densities obtained by RB-SHMC and RBMC-fixed respectively. The blue curve is the equilibrium semicircle law \eqref{eq:semicircle}. The empirical density obtained by RB-SHMC approaches faster than that of RBMC-fixed.
\label{fig:Empirical density DBm}
\end{figure}

The same evolution time \eqref{eq:EvoTime} is used again for this example. 
Figure \ref{fig:Empirical density DBm} presents the empirical densities obtained by running RB-SHMC (red dots) for $10^6$ and $10^7$ sampling iterations, corresponding to 7.6 and 25.6 units of evolution time respectively. The numerical results generated by running RBMC (black triangles) for the same evolution time are also presented for comparison. Obviously, RB-SHMC approaches the equilibrium much faster than RBMC which is consistent with our heuristic justification in Section \ref{subsec:L}.

\begin{figure}[htb]
\centering
\begin{subfigure}[t]{0.48\textwidth}
\centering
\includegraphics[width=\textwidth, height=0.69\textwidth]{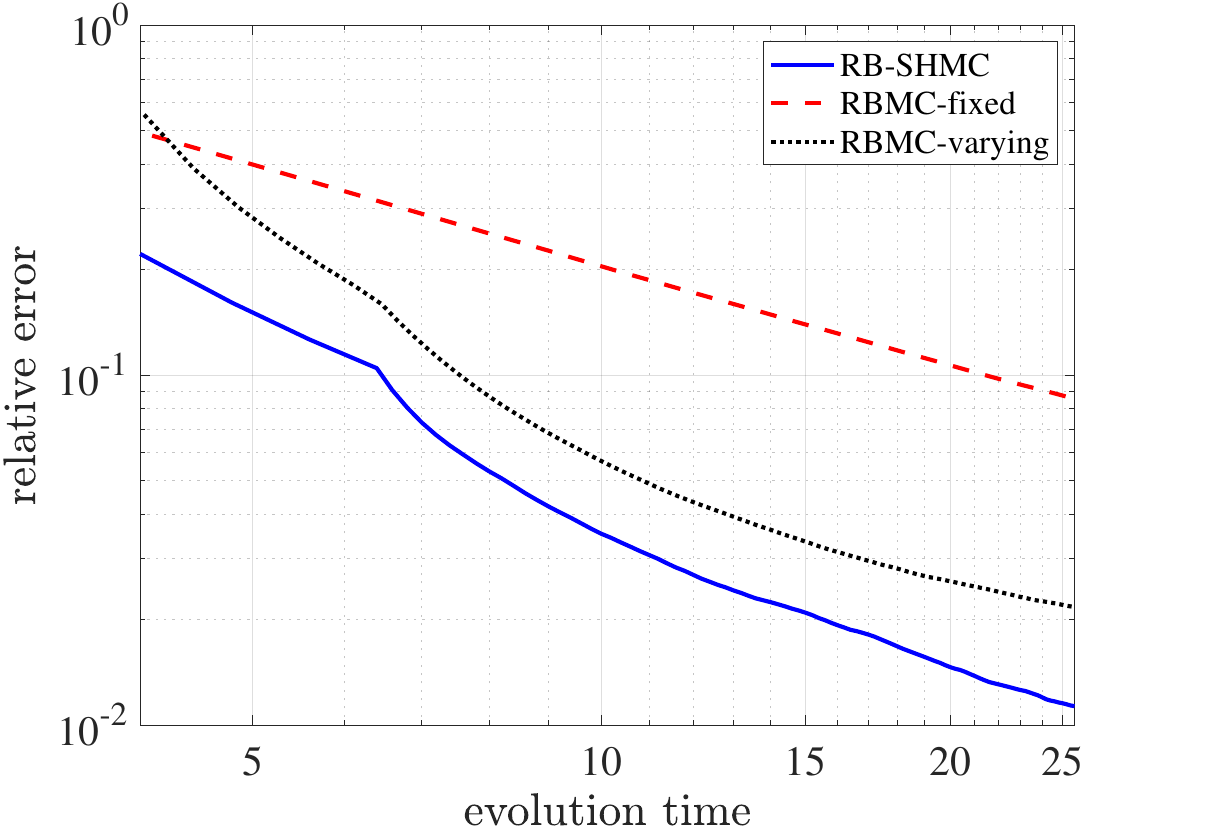}
\end{subfigure}
\hfill
\begin{subfigure}[t]{0.48\textwidth}
\centering
\includegraphics[width=\textwidth]{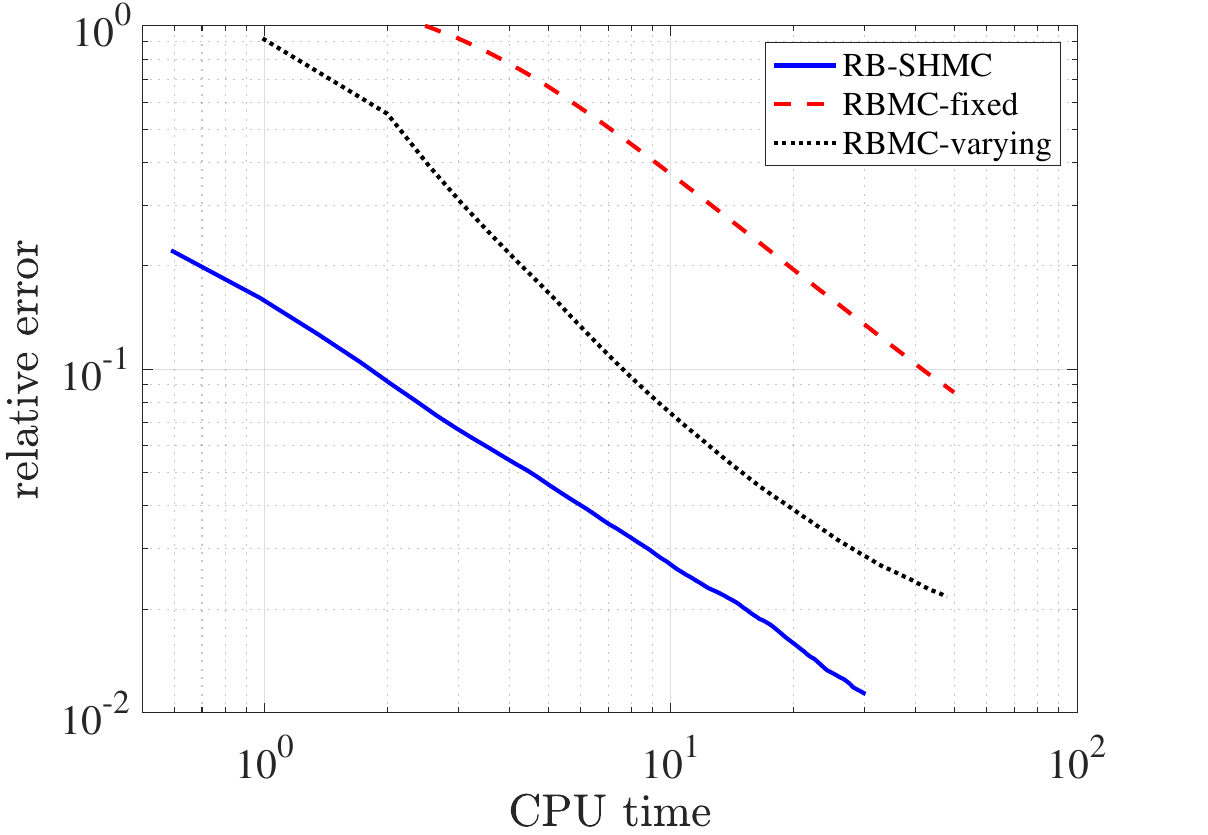}
\end{subfigure}
\caption{Left panel: relative error versus evolution time. Right panel: relative error versus CPU time. In both panels, the blue curve is for RB-SHMC, the red dashed curve is for RBMC-fixed while the black dotted curve is for RBMC-varying. The relative error of the empirical density obtained by RB-SHMC (blue curve) decays faster than those of the other two versions of RBMC considered in this example. }
\label{fig:error of RBHMC, RBMC and RBMC-v2}
\end{figure}

In Figure \ref{fig:error of RBHMC, RBMC and RBMC-v2}, we visualize the relative error of the three methods vs. the evolution time and CPU time respectively. The same ``relative error'' as in Example \ref{subsec:test examples} has been used. 
Clearly, RB-SHMC is far more efficient than RBMC and is also superior to RBMC-v2.

\subsection{Bimodal distribution}\label{subsec:bimodal}

The two examples given below will show that an appropriate potential splitting effectively prevents the problem of generating biased samples when the target distribution is multimodal. These examples showcase the advantage of SHMC and RB-SHMC even when there is no singularity in the potential, further expanding the application territory of the framework of SHMC and RB-SHMC where they can be superior to many other existing methods.

Intuitively, if we can ``flatten'' the landscape of $U$, then the samples can escape from one local minimum, cross the barrier and visit other local minima more efficiently.  
Without loss of generality, we only consider the bimodal distributions for simplicity. 

\subsubsection{Double well potential}\label{subsubsec:double well}

Consider the following one-dimensional double well potential
 \begin{equation}\label{eq:1D double well potential}
 	U(x) = \frac{H}{W^4} (x^2 - W^2)^2,
 \end{equation}
 where $H = 20/\beta$ is the height of the barrier and $W = 1$ is the half width between the two wells. 
 Clearly, the Gibbs distribution corresponding to \eqref{eq:1D double well potential} is 
  \begin{equation}\label{eq:1D double well distribution}
 	\mu(x) = \frac{1}{Z} \exp{\left[-\beta\frac{H}{W^4} (x^2 - W^2)^2\right]},
 \end{equation}
 where $Z$ is the normalizing constant.

\begin{figure}[htb]
\centering
\begin{subfigure}[t]{0.48\textwidth}
\centering
\includegraphics[width=\textwidth]{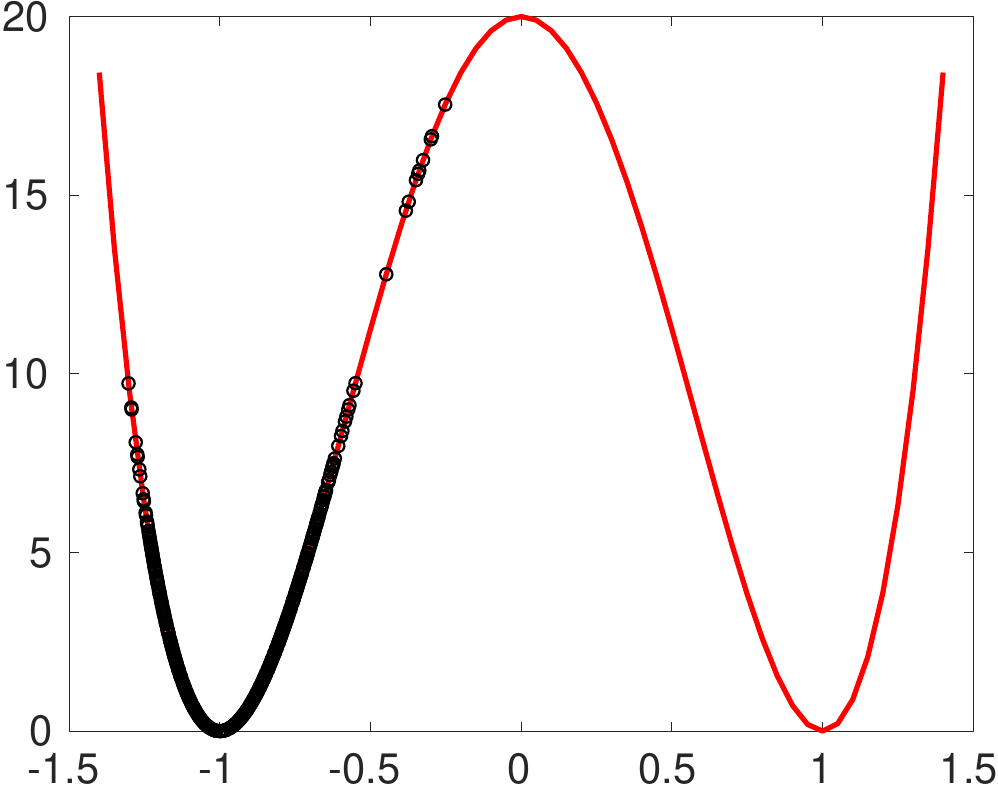}
\subcaption*{Samples}
\end{subfigure}
\hfill
\begin{subfigure}[t]{0.48\textwidth}
\centering
\includegraphics[width=\textwidth]{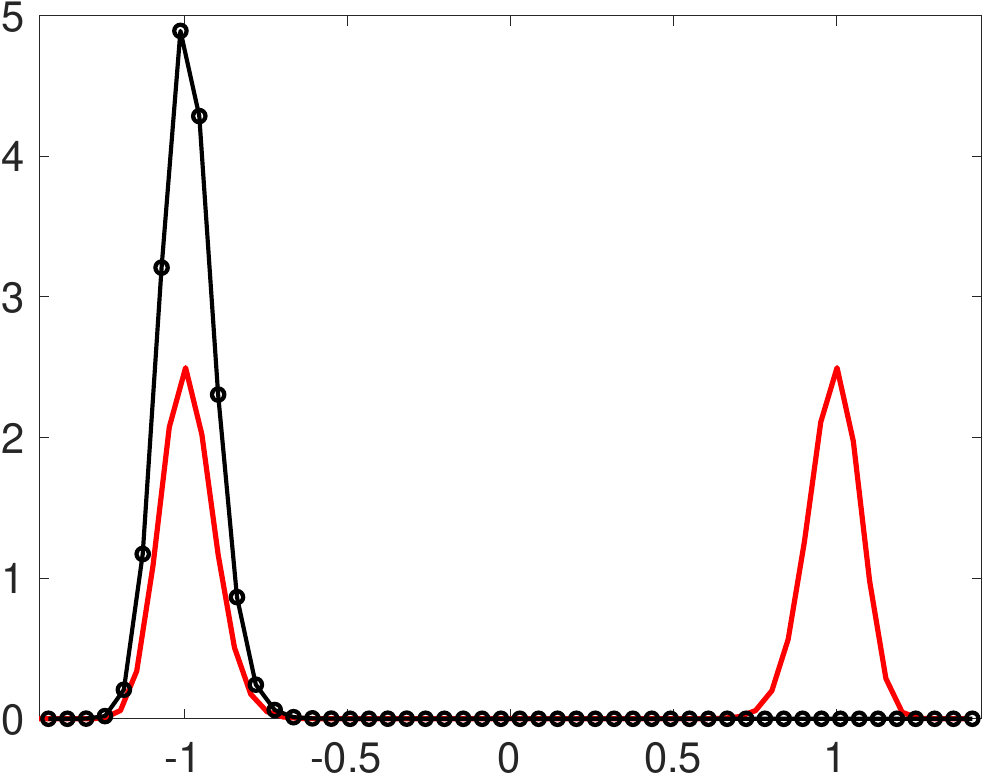}
\subcaption*{Empirical density}
\end{subfigure}
\caption{Samples and the empirical density obtained by HMC with $L = 40$ leapfrog steps and time step size $\Delta t = 2W/L = 0.05$ in each iteration. In the left panel, the black dots are the samples obtained by HMC while the red curve is the exact potential \eqref{eq:1D double well potential}. In the right panel, the black dots are the empirical density obtained by HMC while the red curve is the analytic density \eqref{eq:1D double well distribution}. }
\label{fig:oneD potential & distribution-HMC}
\end{figure}

We first apply the basic HMC algorithm to sample \eqref{eq:1D double well distribution}. 
There are $L = 40$ leapfrog steps in each iteration and the time step is $\Delta t = 2W/L = 0.05$. 
The initial position of the samples is generated randomly on $[-W,W]$. 
Drawing $10^5$ samples from \eqref{eq:1D double well distribution}, the results plotted in Figure \ref{fig:oneD potential & distribution-HMC} clearly show that samples are trapped in the local minimum $x = -W$. 

\begin{figure}[htb]
\centering
\begin{subfigure}[t]{0.48\textwidth}
\centering
\includegraphics[width=\textwidth]{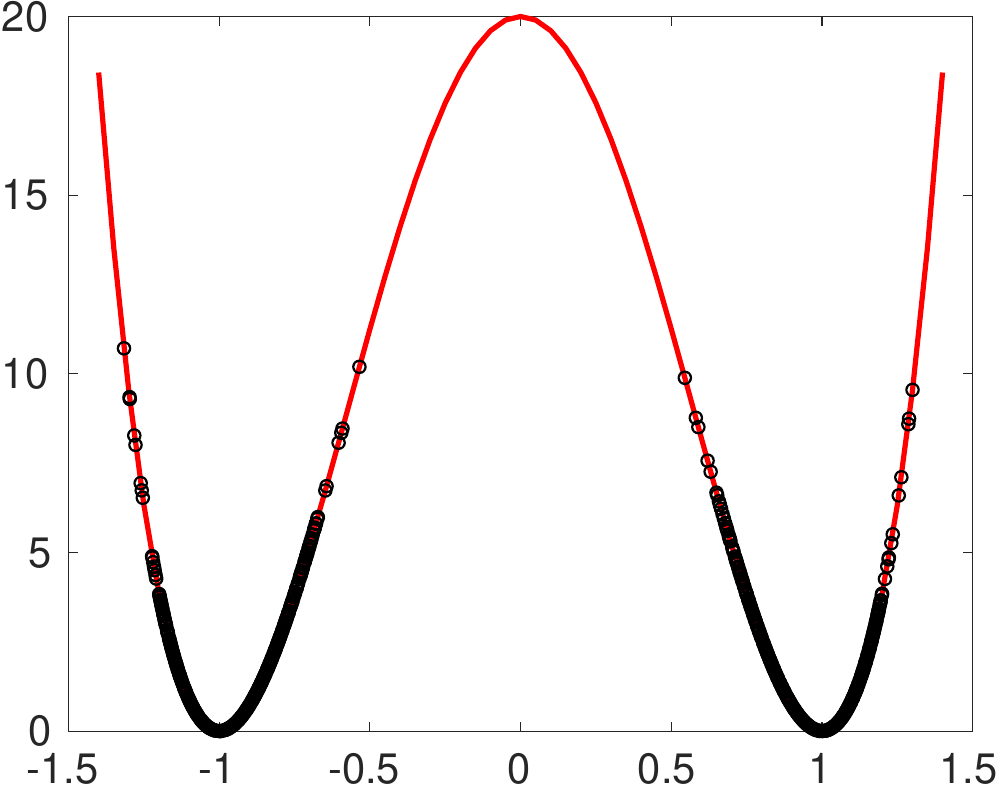}
\subcaption*{Samples}
\end{subfigure}
\hfill
\begin{subfigure}[t]{0.48\textwidth}
\centering
\includegraphics[width=\textwidth]{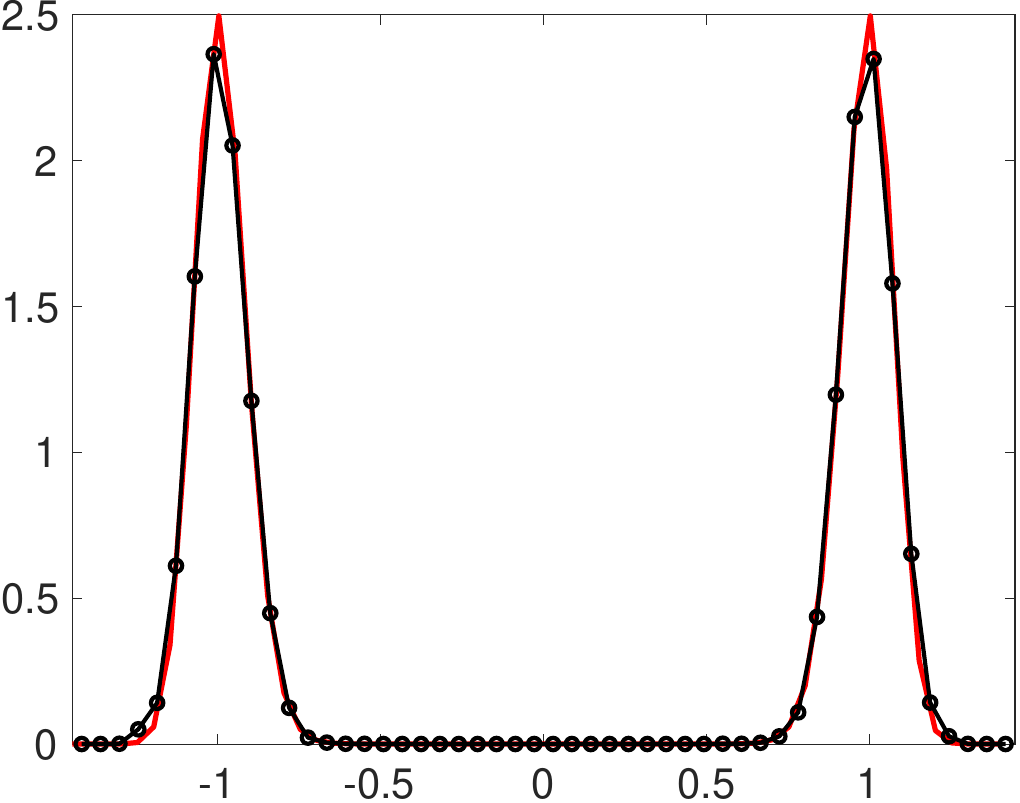}
\subcaption*{Empirical density}
\end{subfigure}
\caption{Samples and the empirical density obtained by SHMC using the surrogate potential $U_1$ given by \eqref{eq:DW_u1} with $\lambda = 0.05$. There are $L = 40$ leapfrog steps with time step size $\Delta t = 2W/L = 0.05$ in each iteration. In the left panel, the black dots are the samples obtained by SHMC while the red curve is the exact potential \eqref{eq:1D double well potential}. In the right panel, the black dots are the empirical density obtained by SHMC while the red curve is the analytic density \eqref{eq:1D double well distribution}. }
\label{fig:oneD potential & distribution-SHMC}
\end{figure}

To fix the above problem, we decompose $U$ into the sum of $U_1$, a double well potential energy with lower barrier to flatten the landscape and $U_2$, the difference between $U$ and $U_1$. Specifically, we choose
\begin{eqnarray}\label{eq:DW_u1}
U_1(x) = \begin{cases}
				\lambda U(x),  |x| < W\\
				U(x),  |x| \ge W 
				\end{cases}
\end{eqnarray}
and
\begin{eqnarray}\label{eq:DW_u2}
U_2(x) = \begin{cases}
					(1-\lambda) U(x),  |x| < W\\
					0,  |x| \ge W 
				\end{cases}
\end{eqnarray}
where $\lambda \in (0, 1)$ is a positive constant being small enough. Here, we choose $\lambda = 0.05$. 
We apply SHMC to draw $10^5$ samples from \eqref{eq:1D double well distribution}. 
We again set $L = 40$ leapfrog steps with time step $\Delta t = 2W/L$ in each iteration. 
The samples we obtained and the resulting empirical density are presented in Figure \ref{fig:oneD potential & distribution-SHMC}. Clearly, the splitting strategy can successfully overcome the local barrier. 

\subsubsection{Gaussian mixture model}\label{subsubsec:gMM}

As can be seen from the previous example, an appropriate splitting easily alleviates the problem of generating biased samples in the vanilla HMC algorithm when we have some knowledge about the landscape, e.g. where the local minima are located. However, in many problems including problems in Bayesian inference, the locations of the local minima are generally unknown. In this section, we consider such examples and propose a potential solution by modifying SHMC and RB-SHMC.

To further illustrate our proposed solution, we slightly modify the first example in \cite[Section 5.1]{welling2011bayesian}, which is a classical Bayesian inference problem of mixture models. We will evaluate the performance of sampling from the posterior distribution using SHMC and RB-SHMC. 
To be specific, we consider the following Gaussian mixture model:
\begin{gather}\label{eq:2D-GMM}
\begin{split}
& \theta_1 \sim \cN(0, \sigma_1^2),\, \theta_2 \sim \cN(0, \sigma_2^2), \\
& y_i \sim \frac{1}{2}\cN(\theta_1, \sigma_y^2) + \frac{1}{2}\cN(\theta_1+\theta_2, \sigma_y^2),
\end{split}
\end{gather}
where $\sigma_1^2 = 10, \, \sigma_2^2 = 1$ and $\sigma_y^2 = 0.5$. 
$N = 100$ data points are drawn from the model with $(\theta_1,\theta_2) = (0, 2)$. 
The left panel of Figure \ref{fig:U & MarDist} shows the potential $U$ corresponding to the target posterior distribution, from which one can imagine that the barrier between the two wells of $U$ will prevent samples from moving between these two wells. Indeed, the samples generated by HMC are trapped in one well; see the left panel of Figure \ref{fig:samples of HMC, SHMC and SRBHMC}.

Inspired by the metadynamics approach \cite{laio2002escaping}, we add scaled Gaussian kernels in the wells to raise the altitudes of the landscape and obtain a new potential energy $U_1$. 
Defining $U_2 := U - U_1$, we get a decomposition $U = U_1 + U_2$. 
The locations of the wells are approximated by $\text{mode}^{mar} (\theta_{1})$ and $\text{mode}^{mar} (\theta_{2})$, the modes\footnote{One may obtain the modes of the marginal distributions using any off-the-shelf mode-finding methods: e.g. the function `findpeaks()' in Matlab.} of the marginal distributions of $\theta_1$ and $\theta_2$ (see the right panel of Figure \ref{fig:U & MarDist}), while $h_{b}$, the height of the barrier in $U$, is estimated by the difference between the altitude at the mid-point between the two wells and the average altitudes between the two wells. 
From our experiments, the scaled Gaussian kernel with height $h_G = h_{b} + 10/\beta$ and covariance matrix $I_2$ turns out to be a good choice for this specific example. 
In particular, 
\begin{gather}
U_1(\theta) = U(\theta) + 2\pi h_G\left[\cN(\text{mode}^{mar} (\theta_{1}), I_2) + \cN(\text{mode}^{mar} (\theta_{2}), I_2)\right].
\end{gather}

\begin{figure}[htbp]
\centering
\begin{subfigure}[t]{0.48\textwidth}
\centering
\includegraphics[width=\textwidth]{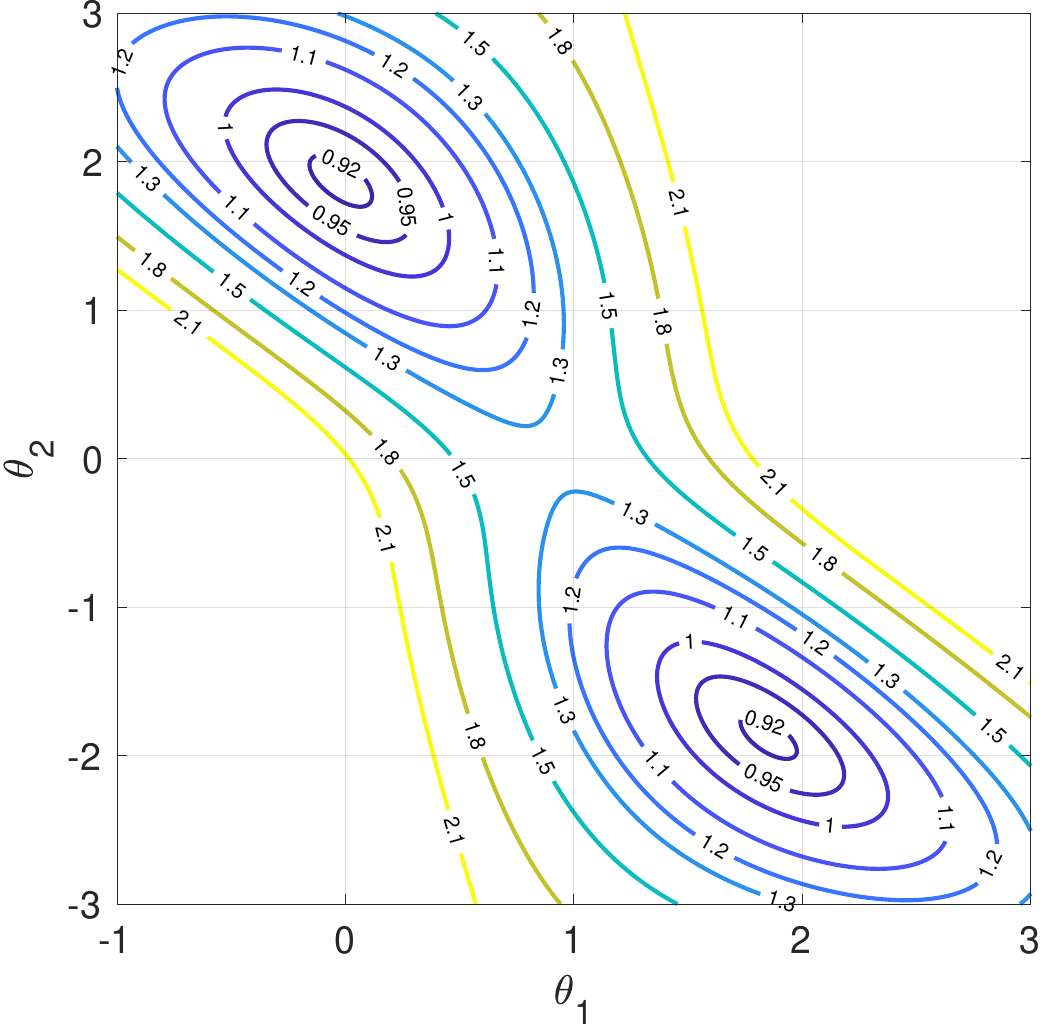}
\subcaption*{Contour plot of $U$}
\end{subfigure}
\hfill
\begin{subfigure}[t]{0.48\textwidth}
\centering
\includegraphics[width=\textwidth]{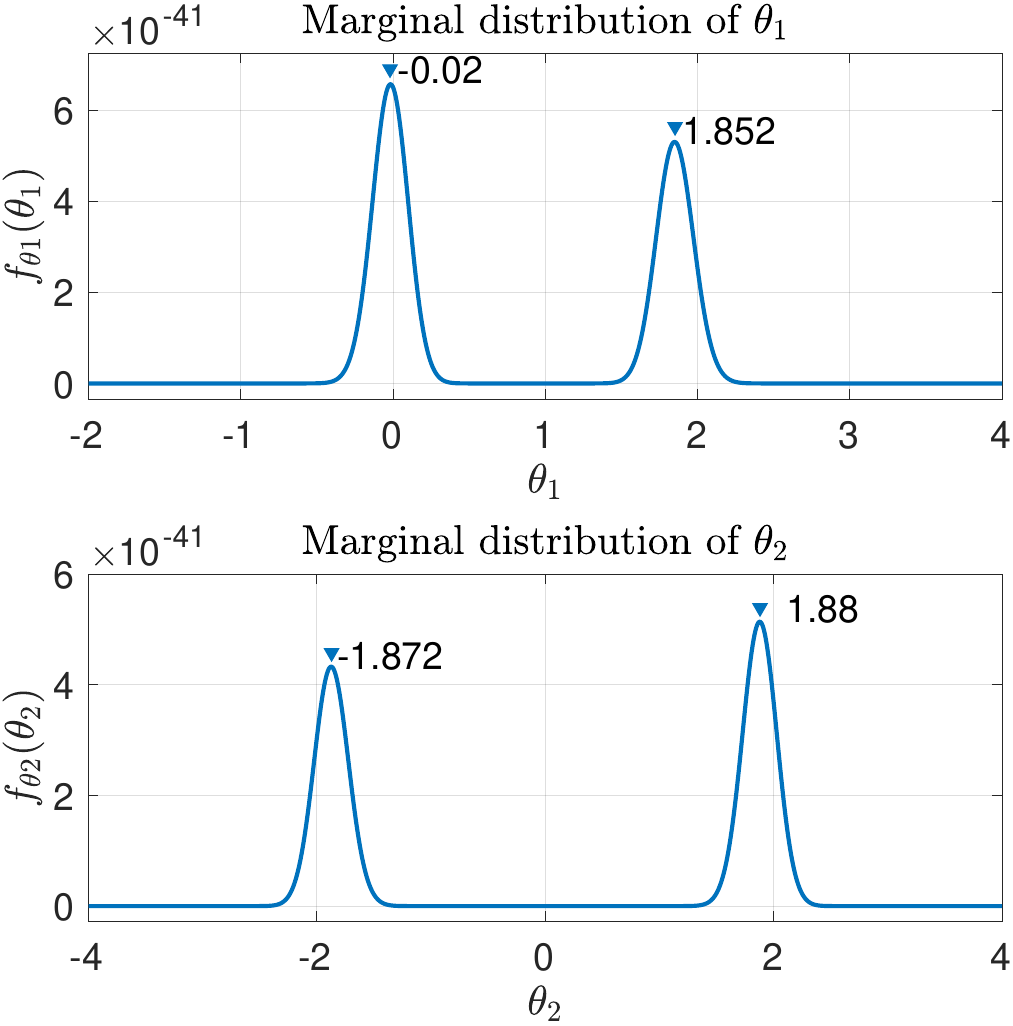}
\subcaption{Marginal distributions}
\end{subfigure}
\caption{Left panel: The contour plot of the potential corresponding to the posterior of model \eqref{eq:2D-GMM}. Right panel: The marginal distributions of the target posterior distribution. }
\label{fig:U & MarDist}
\end{figure}

We choose $\beta = N$ such that the scaling factor of the big summation in $U$ is $1/N$: 
\begin{gather}
U(\theta) = \frac{1}{N}\left(\frac{\theta_1^2}{2\sigma_1^2} + \frac{\theta_2^2}{2\sigma_2^2}\right) - \frac{1}{N}\sum_{i = 1}^{N}\log\left[\exp\left(-\frac{(\theta_1-y_i)^2}{2\sigma_x^2}\right) + \exp\left(-\frac{(\theta_1+\theta_2-y_i)^2}{2\sigma_x^2}\right)\right].
\end{gather}
Under the specific setting of our experiment, the distance between the two wells in $U$ is $d_w = 4.1931$ and the height of the barrier is $h_b = 0.4054$.  The energy barriers for two wells are reduced to 0.0131 and 0.0113 respectively after the above two scaled Gaussian kernels are added to the wells. 
Fix the evolution time of each iteration $L\Delta t = 0.4d_w$ and time step size $\Delta t = 0.01, 0.001$, we apply HMC, SHMC and RB-SHMC with batch size $s = 10$ to sample from the posterior. 
For each method, we collect $10^4$ samples after $10^3$ burn-in iterations. 
The scatter plot of the samples obtained by these three methods are shown in Figures \ref{fig:samples of HMC, SHMC and SRBHMC}. 
Clearly, SHMC and RB-SHMC can sample from the two modes of the posterior without multiple initializations while the samples of HMC are trapped in one single well. 
For HMC, we even run the algorithm with longer evolution time per iteration $L \Delta t = 2d_w$. But still, the samples generated by HMC fail to escape from a single well. 
We have also tried to use SGLD \cite{welling2011bayesian}, the naive version of SGHMC \cite{chen2014stochastic} and split HMC with ``the splitting of data'' strategy in \cite{shahbaba2014split}; however, all of them yield severely biased samples trapped in one single well since both the distance between the two wells and the height of the barrier between the two wells are large (larger than those in the similar experiment conducted in \cite{welling2011bayesian}). 
In summary, one can see that a suitable splitting strategy is powerful in the problem of sampling from a bimodal distribution. 

\begin{figure}[htbp]
\centering
\begin{subfigure}[t]{0.32\textwidth}
\centering
\includegraphics[width=\textwidth]{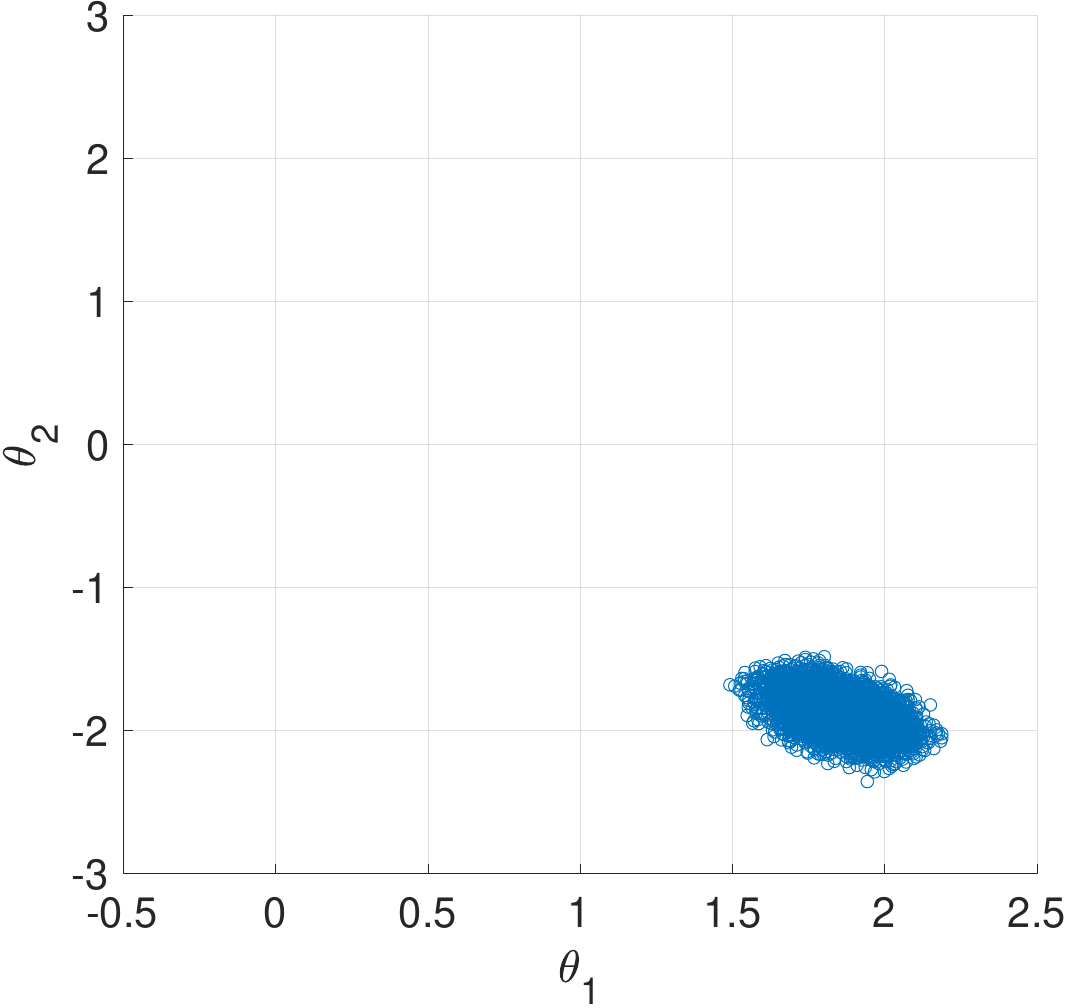}
\subcaption*{Samples of HMC}
\end{subfigure}
\hfill
\begin{subfigure}[t]{0.32\textwidth}
\centering
\includegraphics[width=\textwidth]{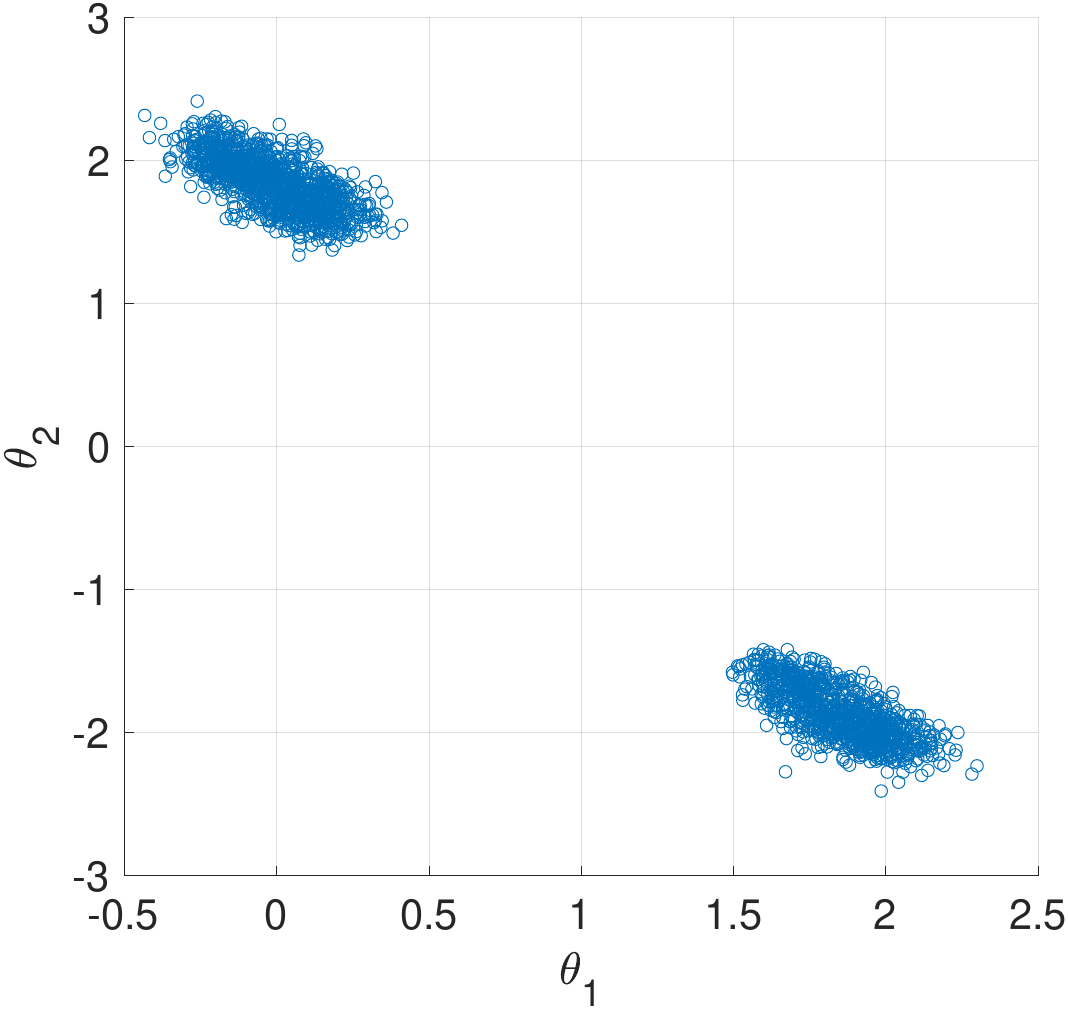}
\subcaption*{Samples of SHMC}
\end{subfigure}
\hfill
\begin{subfigure}[t]{0.32\textwidth}
\centering
\includegraphics[width=\textwidth]{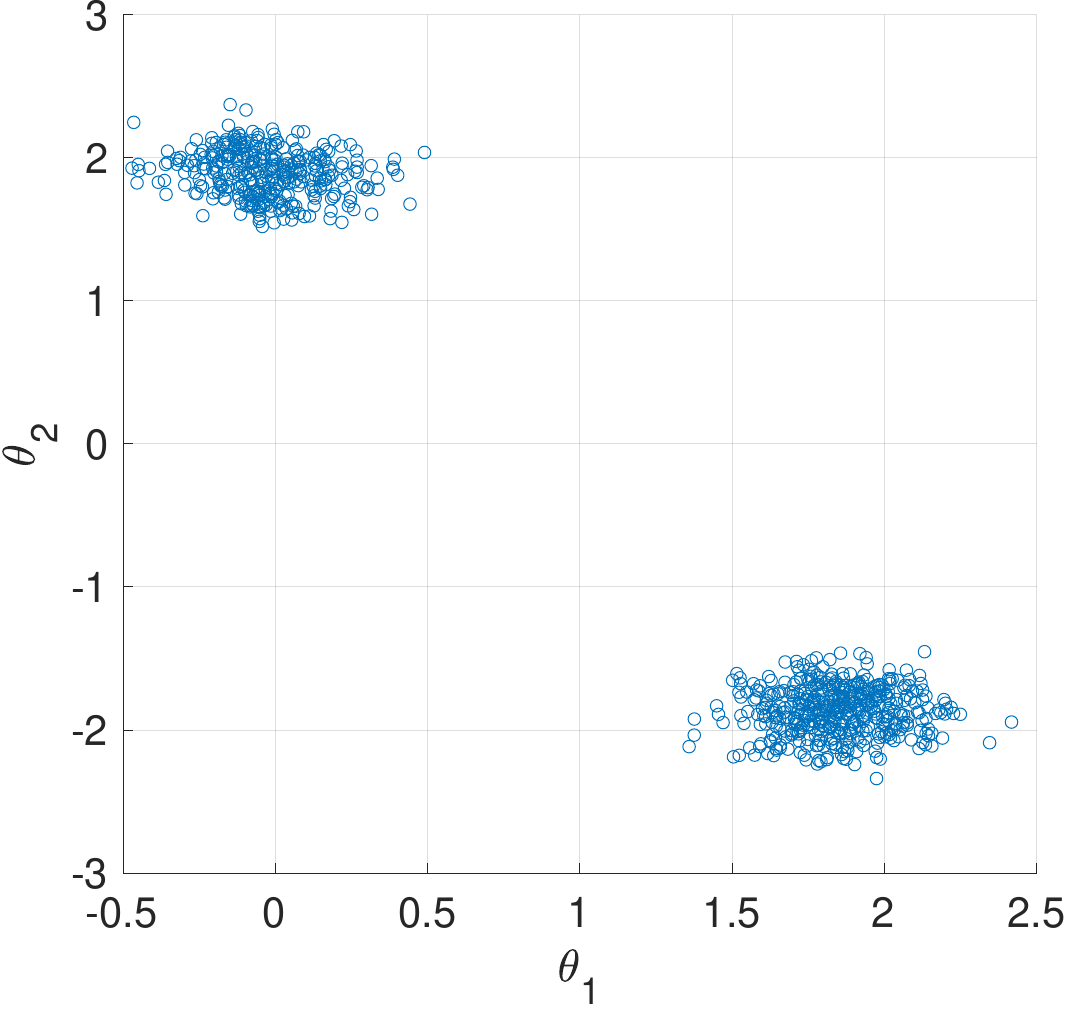}
\subcaption*{Samples of RB-SHMC}
\end{subfigure}
\caption{Samples from the posterior distribution obtained by HMC, SHMC and RB-SHMC with batch size $s = 10$. For all the 3 methods, the time step is fixed to be $\Delta t = 0.01$ and evolution time per iteration is $0.4d_w$.}
\label{fig:samples of HMC, SHMC and SRBHMC}
\end{figure}

Table \ref{table: CPU&acc-2dGMM} shows the average CPU time and acceptance rate of five runs of SHMC and RB-SHMC, with the corresponding standard deviations recorded in the parentheses. 
The CPU time spent in the sampling phase of RB-SHMC is less than $40\%$ of that of SHMC and $t_g$, the CPU time to evaluate the summation term in $\nabla U_1$, of RB-SHMC exhibits a significant advantage over that of SHMC, which clearly shows the benefit of random batch in terms of computational efficiency. 
In the experiment with $\Delta t = 0.01$, the acceptance rate of RB-SHMC is comparatively low since the difference between the exact dynamics and the dynamics using random batch has a relatively large variance when $\Delta t$ is not small enough \cite{li2020random}. 
Nevertheless, the acceptance rate of RB-SHMC is comparable to that of SHMC when the time step is sufficiently small ($\Delta t = 0.001$). 

\begin{table}[htbp]
\centering
\begin{tabular}{|c|c|c|c|c|}
\hline 
  & & $t_g$ (s) & sampling time (s) & acceptance rate \\ 
\hline 
\multirow{2}{*}{$\Delta t = 0.01$} & SHMC & 28.68 (0.4597) & 36.00 (0.5848) & 0.2411 (0.0070) \\ 
\cline{2-5}
 						& RB-SHMC & 6.62 (0.1006) & 14.31 (0.1730) & 0.0856 (0.0018) \\ 
 \hline 
\multirow{2}{*}{$\Delta t = 0.001$} & SHMC & 289.56 (5.3467) & 361.72 (6.8459) & 0.2386 (0.0045) \\ 
\cline{2-5}
						& RB-SHMC & 65.00 (1.0334) & 139.62 (2.7512) & 0.1903 (0.0010) \\ 
\hline 
\end{tabular} 
\caption{The CPU time of evaluating the summation term in $\nabla U_1$, denoted by $t_g$, }the CPU time of the sampling phase and the acceptance rate of SHMC and RB-SHMC. The evolution time per iteration is fixed to be $L\Delta t = 0.4d_w$ and $L = 168$ and $L = 1678$ leapfrog steps are perform respectively for the two choices of $\Delta t$. 
\label{table: CPU&acc-2dGMM}
\end{table}

Finally, we remark that the application of random batch does inject some noise into the original dynamics. As a result, the invariant measure cannot be preserved exactly. In particular, the random batch approximation may not be accurate enough for some complicated problems due to this extra variance by randomization. Hence, we suggest that practitioners exploit the random batch strategy when computational cost is high whereas the demand on accuracy is relatively low. A theoretical understanding of such computational-accuracy trade-off even for general random batch methods will be an interesting research direction to explore in future works.

\subsection{Lennard-Jones fluids}\label{subsec:lj}

Now we apply RB-SHMC to sample from the Gibbs distribution of a three-dimensional interacting particle system whose interaction potential is modelled by the Lennard-Jones potential
\begin{gather}\label{eq:potential_LJ}
\phi(r_i - r_j) = \Phi(r_{ij}) = 4\left[\left(\frac{1}{r_{ij}}\right)^{12} - \left(\frac{1}{r_{ij}}\right)^{6}\right],
\end{gather}
where $r_{ij} = |r_i - r_j|$ is the distance between the two particles located at $r_i$ and $r_j$ in $\bbR^3$. Conventionally, we use periodic boxes with side length $L$ to approximate the fluids in the numerical simulation. Therefore, each particle $i$ interacts with not only all the other particles $j \neq i$ in the system but also their periodic images (including its own periodic image). The corresponding Gibbs distribution is given by
\begin{gather}\label{eq:Gibbs_LJ}
\mu \propto \exp\left[-\frac{\beta}{2} \sum_{n}{'}\sum_{i,j}\phi(r_i - r_j+nL)\right],
\end{gather}
where $\beta=1/T$ is the inverse temperature (all quantities have been scaled according to some basic units so that they are dimensionless) and $n \in \bbZ^3$. Here, we introduce the notation $\sum_n{'}$ to mean that the term with $n = 0$ is excluded from the outer summation if $i = j$.
Note that the Lennard-Jones potential \eqref{eq:potential_LJ} decays rapidly as the distance between the two particles increases so that one may approximate the discrete sum by the continuous integral over $r_{i}, r_{j}: |r_i-r_j+nL|\ge R_c$ for some chosen cutoff $R_c>0$. Thus the pressure of the system can be approximated by
\begin{gather}\label{eq:pressure_for_periodic}
P=\frac{\rho}{\beta}+\frac{8}{V}\sum_{i=1}^N \sum\limits_{\substack{j: j>i,\\r_{ij}^*<R_c}}\left[2 \left(\frac{1}{r_{ij}^*}\right)^{12}-\left(\frac{1}{r_{ij}^*}\right)^{6}\right]+\dfrac{16}{3}\pi\rho^2 \left[\frac{2}{3}\left(\frac{1}{R_c}\right)^9-\left(\frac{1}{R_c}\right)^3\right],
\end{gather}
where $\rho=N/L^3$ is the particle density and $r_{ij}^*=|r_i-r_j+n_{ij}^* L|$ is the distance between particle $i$ and the nearest image of $j$.
One may refer to \cite{frenkel2001understanding, li2020random} for more details.
Moreover, with this cutoff, we only need to consider the interactions between particles and images that have distance less than $R_c$ during our sampling algorithms using RB-SHMC and RBMC.

In this example, we take the cutoff length
\[
R_c=L/2
\]
and compare our proposed RB-SHMC algorithm with RBMC \cite{li2020random}. We use similar settings to those in \cite{li2020random}. Specifically, the system consists of $N = 500$ particles and the length of the periodic boxes is varied with the given density $\rho \in [0, 1]$ by
\begin{gather}
L = \left(\frac{N}{\rho}\right)^{1/3}.
\end{gather}
We adopt a simpler splitting strategy than that in \cite{li2020random}, given by
\begin{gather}\label{eq:LJsplit}
\Phi(r) = \Phi_1(r) + \Phi_2(r),
\end{gather}
with
\begin{eqnarray}\label{eq:LJ_u1}
\Phi_1(r) = \begin{cases}
					-2^{-1/6}r, & 0 < r < 2^{1/6},\\
					4\left(r^{-12} - r^{-6}\right), & r \geq  2^{1/6},
			\end{cases}
\end{eqnarray}
and
\begin{eqnarray}\label{eq:LJ_u2}
\Phi_2(r) = \begin{cases}
					4\left(r^{-12} - r^{-6}\right) +2^{-1/6}r & 0 < r <  2^{1/6},\\
					0, & r \geq  2^{1/6}.
			\end{cases}
\end{eqnarray}

\begin{figure}
\centering
\includegraphics[width=0.6\textwidth]{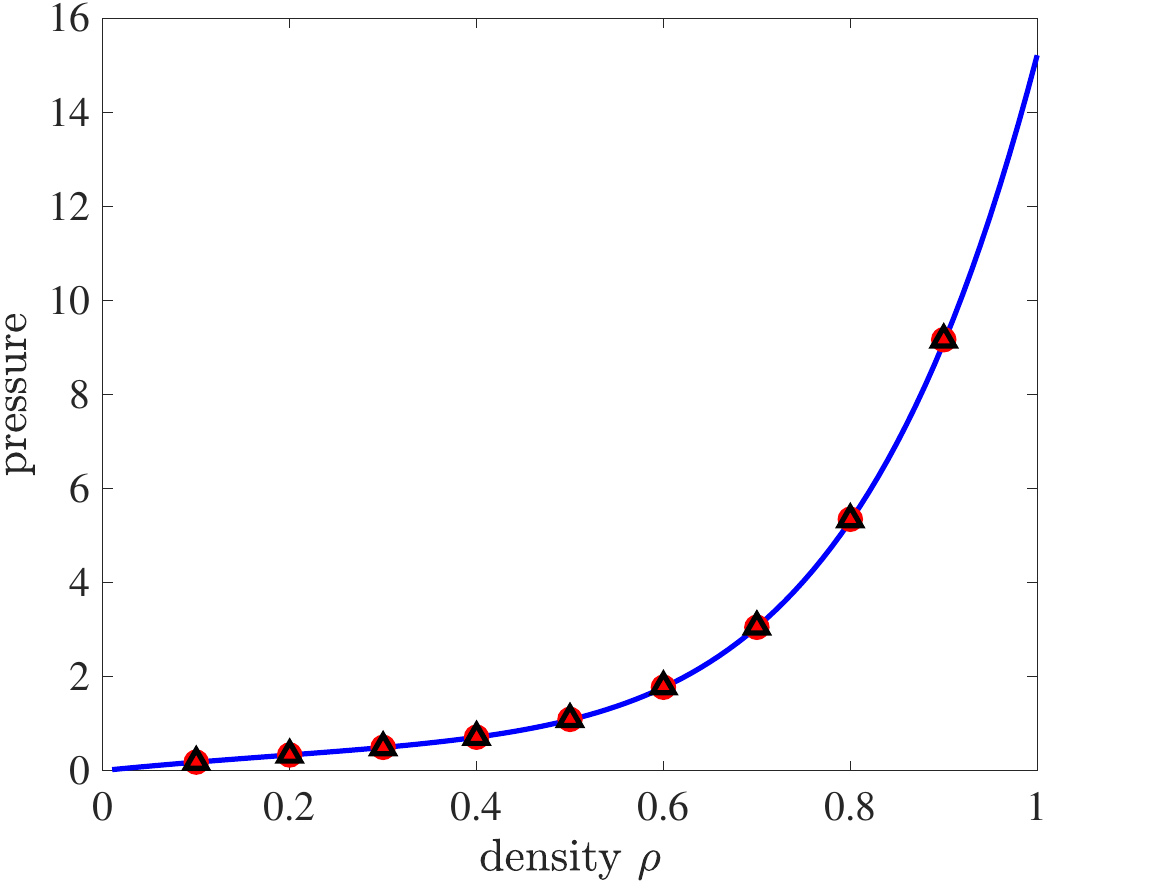}
\caption{The pressure obtained by $10^7$ iterations of RB-SHMC (red dots) and RBMC (black triangles) respectively. The samples are collected after the burn-in phase, defined as the first $2\times 10^5$ iterations of the two methods. The blue curve is the reference solution given by the fitting curve in \cite{johnson1993lennard}.}
\label{fig:pressure}
\end{figure}

In the experiment, we let the number of leapfrog steps and the timestep for RB-SHMC gradually decrease. Specifically, the number of leapfrog steps is chosen to be \[
L_n = \max\left\{\left\lceil \frac{\log2\cdot L_1}{\log(1+k_n)}\right\rceil, L_1 - k_n + 1, 3 \right\}
\]
with $k_n = \left\lceil \frac{n}{2000}\right\rceil$ and the initial value $L_1 = 20$. By this strategy, $L_n$ decreases as $21-n/2000$ in the earlier stage of iterations (roughly when $n \leq 3.2 \times 10^{4}$) and as $20 \log2/\log(1+n/2000)$ in the later stage (roughly when $n \geq 3.2 \times 10^{4}$). Moreover, we make sure $L_n \geq 3$ so that the displacement is not too small.
The timestep is chosen to be $\Delta t_n = \Delta t_1 n^{-\gamma}$ with the initial step size $\Delta t_1 = 0.2$ and the decaying exponent $\gamma = 0.06$. 
Following \cite{li2020random}, the number of discretization steps and timestep size of RBMC are fixed at $L_n \equiv 9$ and $\Delta t_n \equiv 0.01$ in each iteration. For both RB-SHMC and RBMC, the batch size is chosen to be $s = 1$ and the first $2\times 10^5$ iterations are taken to be the burn-in phase. Only the samples generated after the burn-in phase are collected. The pressure obtained over $M$ iterations is computed in the following way: we first compute the pressure \eqref{eq:pressure_for_periodic} of each sample configuration collected after the burn-in phase, and then we take the sample average of the pressures over these $M$ iterations.
It can be seen from Figure \ref{fig:pressure} that the pressures obtained by $10^7$ iterations of both RB-SHMC (red dots) and RBMC (black triangles) capture the reference solution (blue curve) given by the fitting curve provided in \cite{johnson1993lennard} well. 

Following \cite{li2020random}, the numerical results are evaluated by the relative error in the $\ell^2$ norm: 
\begin{gather}\label{eq:error_LJ}
\left.\sqrt{\sum\limits_{i=1}\limits^{9}\frac{1}{9}(P_i-P(\rho_i))^2} \middle/\sqrt{\sum_{i=1}^9 \frac{1}{9}P(\rho_i)^2}\right.,
\end{gather}
where $P_i$ denotes the pressure at different density values $\rho =\rho_i$ (here we choose 9 different density values as shown in Figure \ref{fig:pressure}) and $P(\cdot)$ represents the reference solution given by the fitting curves in \cite{johnson1993lennard}. Figure \ref{fig:LJ_error} presents how the relative errors of RB-SHMC and RBMC vary over the CPU time used. Here, the value of the pressure $P_i$ used to compute the relative error \eqref{eq:error_LJ} is averaged over the pressures generated by 10 independent chains to further reduce the impact of randomness. One can clearly see that both the relative errors of RB-SHMC and RBMC quickly decrease to below $1\%$ and the relative error of RB-SHMC decays even faster than that of RBMC. \\

\begin{figure}
\centering
\includegraphics[width=0.6\textwidth]{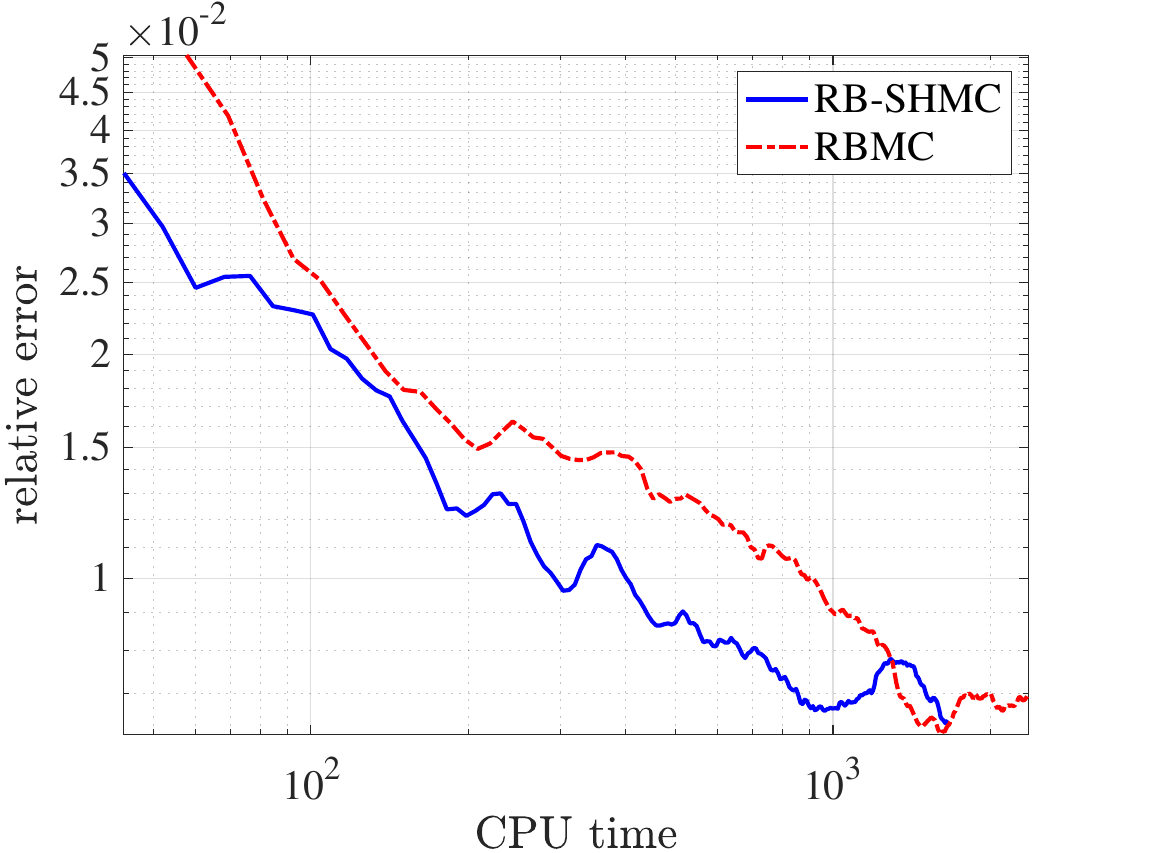}
\caption{The relative $\ell^2$ error of the pressures obtained by RB-SHMC (blue curve) decays faster than that of RBMC (red dashed curve). The value of the relative error \eqref{eq:error_LJ} is computed by the pressure averaged over 10 independent chains.}
\label{fig:LJ_error}
\end{figure}

\begin{remark}
The three examples arising from interacting particle systems (Example \ref{subsec:test examples}, Example \ref{subsec:DBm} and Example \ref{subsec:lj}) clearly exhibit the benefits of RB-SHMC. First, for an interacting particle system with $N$ particles, the application of the random batch strategy efficiently reduces the computational cost per timestep from $\cO(N)$ to $\cO(1)$. 
A second benefit of RB-SHMC comes from the splitting strategy. Note that the interaction potentials in the examples of Dyson Brownian motion and Lennard-Jones fluids are singular, popular gradient-based sampling methods, such as SGLD and standard HMC, are doomed to fail. 
Moreover, a clever splitting scheme can further reduce the computational cost. Recall that we perform the splitting by truncating the pairwise interaction potential at a small cutoff radius $r = r_c$ so that the number of particles within the range of $r_c$ to the current particle is $\cO(1)$ and thus the computational cost of the Metropolis rejection step is reduced from $\cO(N)$ to $\cO(1)$.
\end{remark}

\section*{Acknowledgements}

The authors would like to thank two anonymous referees for valuable comments which helped to improve our paper. The work of L. Li was partially supported by the National Key R\&D Program of China No. 2020YFA0712000, NSFC Grant No. 11901389, 12031013, and Shanghai Science and Technology Commission Grant No. 20JC144100, 21JC1403700.
The work of L. Liu was partially supported by Shanghai Pujiang Program Research Grant No.20PJ1408900, Institute of Modern Analysis -- A Shanghai Frontier Research Center Grant No. ZXWH2071101/002, Shanghai Municipal Science and Technology Major Project No.2021SHZDZX0102, Shanghai Science and Technology Commission Grant No.21ZR1431000 \& 21JC1402900, NSFC Grant No.12101397 \& 12090024.

\appendix
\section{Proof of Lemma \ref{lem:E P^q}}\label{apdx:pf lem E P^q}

\begin{proof}
Denote $C_{q}$ as a constant depending on $q$ and $C_{q}$ may change from line to line. For any $t\in [0, T]$,
\begin{gather}
\begin{split}
\frac{d}{dt} \mathbb{E}\left(\left|\tbp(t)\right|^q\middle| \mathcal{F}_{L}\right) 
= & \ q \mathbb{E}\left(\left|\tbp(t)\right|^{q-1} \left\vert -\nabla V_1\left(\tbx(t)\right) + \frac{1}{s}\sum_{j\in\xi_{L_t}}\nabla \psi_j\left(\tbx\right)\right\vert \ \middle| \mathcal{F}_{L}\right) \\
\leq & \ q\left(\|\nabla V_1\|_{\infty} + \|\nabla \psi\|_{\infty} \right)\mathbb{E}\left(\left|\tbp(t)\right|^{q-1}\middle| \mathcal{F}_{L}\right) \\
\leq & \ C_q \mathbb{E}\left(\left|\tbp(t)\right|^{q}\middle| \mathcal{F}_{L}\right)^{\frac{q-1}{q}}
\end{split}
\end{gather}
where the last line follows from H\"{o}lder's inequality. By integration, one has
\begin{gather}
\mathbb{E}\left(\left|\tbp(t)\right|^q\middle| \mathcal{F}_{L}\right)
\leq C_q\left(\left|\tbp(0)\right|^q + t^q\right)
\leq C_q\left(\left|\tbp(0)\right|^q + T^q\right),
\end{gather}
and thus equation \eqref{eq:E P^q-global} follows from taking expectation w.r.t. the randomness in $\mathcal{F}_{L}$ on both sides and setting $q = 4$.
\end{proof}

\bibliographystyle{siam}  
\bibliography{SHMCrb_ref}

\begin{thebibliography}{10}

\bibitem{allen1987}
{\sc M.~P. Allen and D.~J. Tildesley}, {\em Computer simulation of liquids},
  Oxford University Press, 1987.

\bibitem{banerjee2014hierarchical}
{\sc S.~Banerjee, B.~P. Carlin, and A.~E. Gelfand}, {\em Hierarchical modeling
  and analysis for spatial data}, CRC press, 2014.

\bibitem{barbu2020hamiltonian}
{\sc A.~Barbu and S.-C. Zhu}, {\em Hamiltonian and {L}angevin {M}onte {C}arlo},
  in Monte Carlo Methods, Springer, 2020, pp.~281--325.

\bibitem{benettin1994hamiltonian}
{\sc G.~Benettin and A.~Giorgilli}, {\em On the {H}amiltonian interpolation of
  near-to-the identity symplectic mappings with application to symplectic
  integration algorithms}, Journal of Statistical Physics, 74 (1994),
  pp.~1117--1143.

\bibitem{besag1994comments}
{\sc J.~E. Besag}, {\em Comments on ``{R}epresentations of knowledge in complex
  systems'' by {U.} {G}renander and {MI} {M}iller}, Journal of the Royal
  Statistical Society: Series B (Methodological), 56 (1994), pp.~549--581.

\bibitem{beskos2013optimal}
{\sc A.~Beskos, N.~Pillai, G.~Roberts, J.-M. Sanz-Serna, and A.~Stuart}, {\em
  Optimal tuning of the hybrid {M}onte {C}arlo algorithm}, Bernoulli, 19
  (2013), pp.~1501--1534.

\bibitem{betancourt2017conceptual}
{\sc M.~Betancourt}, {\em A conceptual introduction to {H}amiltonian {M}onte
  {C}arlo}, arXiv preprint arXiv:1701.02434,  (2017).

\bibitem{bottou1998online}
{\sc L.~Bottou}, {\em Online learning and stochastic approximations}, On-line
  Learning in Neural Networks, 17 (1998), p.~142.

\bibitem{bubeck2015convex}
{\sc S.~Bubeck}, {\em Convex {O}ptimization: {A}lgorithms and {C}omplexity},
  Foundations and Trends{\textregistered} in Machine Learning, 8 (2015),
  pp.~231--357.

\bibitem{chen2014stochastic}
{\sc T.~Chen, E.~Fox, and C.~Guestrin}, {\em Stochastic gradient {Hamiltonian
  Monte Carlo}}, in International Conference on Machine Learning, PMLR, 2014,
  pp.~1683--1691.

\bibitem{chen2020fast}
{\sc Y.~Chen, R.~Dwivedi, M.~J. Wainwright, and B.~Yu}, {\em Fast mixing of
  {Metropolized Hamiltonian Monte Carlo}: Benefits of multi-step gradients},
  Journal of Machine Learning Research,  (2020).

\bibitem{dai2020wang}
{\sc C.~Dai and J.~S. Liu}, {\em Wang-{L}andau algorithm as stochastic
  optimization and its acceleration}, Physical Review E, 101 (2020), p.~033301.

\bibitem{dalalyan2017theoretical}
{\sc A.~S. Dalalyan}, {\em Theoretical guarantees for approximate sampling from
  smooth and log-concave densities}, Journal of the Royal Statistical Society:
  Series B (Statistical Methodology), 79 (2017), pp.~651--676.

\bibitem{dalalyan2019user}
{\sc A.~S. Dalalyan and A.~Karagulyan}, {\em User-friendly guarantees for the
  {L}angevin {M}onte {C}arlo with inaccurate gradient}, Stochastic Processes
  and Their Applications, 129 (2019), pp.~5278--5311.

\bibitem{ding2021random}
{\sc Z.~Ding, Q.~Li, J.~Lu, and S.~Wright}, {\em Random {C}oordinate
  {U}nderdamped {L}angevin {M}onte {C}arlo}, in International Conference on
  Artificial Intelligence and Statistics, PMLR, 2021, pp.~2701--2709.

\bibitem{duane1987hybrid}
{\sc S.~Duane, A.~D. Kennedy, B.~J. Pendleton, and D.~Roweth}, {\em Hybrid
  {M}onte {C}arlo}, Physics Letters B, 195 (1987), pp.~216--222.

\bibitem{durrett2019probability}
{\sc R.~Durrett}, {\em Probability: Theory and Examples}, vol.~49, Cambridge
  University Press, 2019.

\bibitem{dyson1962brownian}
{\sc F.~J. Dyson}, {\em A {B}rownian-motion model for the eigenvalues of a
  random matrix}, Journal of Mathematical Physics, 3 (1962), pp.~1191--1198.

\bibitem{erdos2017dynamical}
{\sc L.~Erdos and H.-T. Yau}, {\em A dynamical approach to random matrix
  theory}, Courant Lecture Notes in Mathematics, 28 (2017).

\bibitem{fraley2002model}
{\sc C.~Fraley and A.~E. Raftery}, {\em Model-based clustering, discriminant
  analysis, and density estimation}, Journal of the American Statistical
  Association, 97 (2002), pp.~611--631.

\bibitem{frenkel2001understanding}
{\sc D.~Frenkel and B.~Smit}, {\em Understanding molecular simulation: From
  algorithms to applications}, vol.~1, Elsevier, 2001.

\bibitem{geman1984stochastic}
{\sc S.~Geman and D.~Geman}, {\em Stochastic relaxation, {G}ibbs distributions,
  and the {B}ayesian restoration of images}, IEEE Transactions on Pattern
  Analysis and Machine Intelligence,  (1984), pp.~721--741.

\bibitem{geyer1991markov}
{\sc C.~J. Geyer}, {\em {Markov Chain Monte Carlo} maximum likelihood}, in
  Computing science and statistics: Proceedings of 23rd Symposium on the
  Interface Interface Foundation, 1991, pp.~156--163.

\bibitem{giordano2020consistency}
{\sc M.~Giordano and R.~Nickl}, {\em Consistency of {Bayesian} inference with
  {Gaussian} process priors in an elliptic inverse problem}, Inverse Problems,
  36 (2020), p.~085001.

\bibitem{green1995reversible}
{\sc P.~J. Green}, {\em Reversible jump {M}arkov chain {M}onte {C}arlo
  computation and {B}ayesian model determination}, Biometrika, 82 (1995),
  pp.~711--732.

\bibitem{hairer2006geometric}
{\sc E.~Hairer, C.~Lubich, and G.~Wanner}, {\em Geometric numerical
  integration: structure-preserving algorithms for ordinary differential
  equations}, vol.~31, Springer Science \& Business Media, 2006.

\bibitem{hastings1970monte}
{\sc W.~K. Hastings}, {\em Monte {Carlo} sampling methods using {Markov} chains
  and their applications}, Biometrika, 57 (1970), p.~97.

\bibitem{hetenyi2002multiple}
{\sc B.~Hetenyi, K.~Bernacki, and B.~J. Berne}, {\em Multiple ``time step''
  {Monte Carlo}}, The Journal of Chemical Physics, 117 (2002), pp.~8203--8207.

\bibitem{jasra2005markov}
{\sc A.~Jasra, C.~C. Holmes, and D.~A. Stephens}, {\em Markov {C}hain {M}onte
  {C}arlo methods and the label switching problem in {B}ayesian mixture
  modeling}, Statistical Science, 20 (2005), pp.~50--67.

\bibitem{jin2021random}
{\sc S.~Jin and L.~Li}, {\em Random {B}atch {M}ethods for classical and quantum
  interacting particle systems and statistical samplings}, arXiv preprint
  arXiv:2104.04337,  (2021).

\bibitem{jin2018random}
{\sc S.~Jin, L.~Li, and J.-G. Liu}, {\em Random batch methods ({RBM}) for
  interacting particle systems}, Journal of Computational Physics, 400 (2020),
  p.~108877.

\bibitem{jin2021convergence}
\leavevmode\vrule height 2pt depth -1.6pt width 23pt, {\em Convergence of the
  random batch method for interacting particles with disparate species and
  weights}, SIAM Journal on Numerical Analysis, 59 (2021), pp.~746--768.

\bibitem{jin2020randomsecondorder}
{\sc S.~Jin, L.~Li, and Y.~Sun}, {\em On the {R}andom {B}atch {M}ethod for
  second order interacting particle systems}, arXiv preprint arXiv:2011.10778,
  (2020).

\bibitem{jin2020randombatchEwald}
{\sc S.~Jin, L.~Li, Z.~Xu, and Y.~Zhao}, {\em A random batch {E}wald method for
  particle systems with {C}oulomb interactions}, arXiv preprint
  arXiv:2010.01559,  (2020).

\bibitem{jin2020randomQMC}
{\sc S.~Jin and X.~Li}, {\em Random batch algorithms for quantum {Monte Carlo}
  simulations}, arXiv preprint arXiv:2008.12990,  (2020).

\bibitem{johnson1993lennard}
{\sc J.~K. Johnson, J.~A. Zollweg, and K.~E. Gubbins}, {\em The
  {L}ennard-{J}ones equation of state revisited}, Molecular Physics, 78 (1993),
  pp.~591--618.

\bibitem{laio2002escaping}
{\sc A.~Laio and M.~Parrinello}, {\em Escaping free-energy minima}, Proceedings
  of the National Academy of Sciences, 99 (2002), pp.~12562--12566.

\bibitem{lan2014wormhole}
{\sc S.~Lan, J.~Streets, and B.~Shahbaba}, {\em Wormhole {H}amiltonian {M}onte
  {C}arlo}, in Proceedings of the AAAI Conference on Artificial Intelligence,
  vol.~28, 2014.

\bibitem{lee2018convergence}
{\sc Y.~T. Lee and S.~S. Vempala}, {\em Convergence rate of {Riemannian
  Hamiltonian Monte Carlo} and faster polytope volume computation}, in
  Proceedings of the 50th Annual ACM SIGACT Symposium on Theory of Computing,
  2018, pp.~1115--1121.

\bibitem{leimkuhler2004simulating}
{\sc B.~Leimkuhler and S.~Reich}, {\em Simulating {H}amiltonian dynamics},
  vol.~14, Cambridge University Press, 2004.

\bibitem{li2020stochastic}
{\sc L.~Li, Y.~Li, J.-G. Liu, Z.~Liu, and J.~Lu}, {\em A stochastic version of
  {Stein} variational gradient descent for efficient sampling}, Communications
  in Applied Mathematics and Computational Science, 15 (2020), pp.~37--63.

\bibitem{li2020random}
{\sc L.~Li, Z.~Xu, and Y.~Zhao}, {\em A random-batch {M}onte {C}arlo method for
  many-body systems with singular kernels}, SIAM Journal on Scientific
  Computing, 42 (2020), pp.~A1486--A1509.

\bibitem{lovasz1993random}
{\sc L.~Lov{\'a}sz and M.~Simonovits}, {\em Random walks in a convex body and
  an improved volume algorithm}, Random Structures \& Algorithms, 4 (1993),
  pp.~359--412.

\bibitem{mangoubi2018does}
{\sc O.~Mangoubi, N.~S. Pillai, and A.~Smith}, {\em Does {H}amiltonian {M}onte
  {C}arlo mix faster than a random walk on multimodal densities?}, arXiv
  preprint arXiv:1808.03230,  (2018).

\bibitem{marinari1992simulated}
{\sc E.~Marinari and G.~Parisi}, {\em Simulated tempering: a new {M}onte
  {C}arlo scheme}, Europhysics Letters, 19 (1992), p.~451.

\bibitem{markowich2000trend}
{\sc P.~A. Markowich and C.~Villani}, {\em On the trend to equilibrium for the
  {F}okker-{P}lanck equation: an interplay between physics and functional
  analysis}, Mathematica Contemporanea, 19 (2000), pp.~1--31.

\bibitem{metropolis1953equation}
{\sc N.~Metropolis, A.~W. Rosenbluth, M.~N. Rosenbluth, A.~H. Teller, and
  E.~Teller}, {\em Equation of state calculations by fast computing machines},
  The Journal of Chemical Physics, 21 (1953), pp.~1087--1092.

\bibitem{miasojedow2013adaptive}
{\sc B.~Miasojedow, E.~Moulines, and M.~Vihola}, {\em An adaptive parallel
  tempering algorithm}, Journal of Computational and Graphical Statistics, 22
  (2013), pp.~649--664.

\bibitem{neal2011mcmc}
{\sc R.~M. Neal}, {\em {MCMC} using {H}amiltonian dynamics}, Handbook of Markov
  Chain Monte Carlo, 2 (2011), p.~2.

\bibitem{nickl2020polynomial}
{\sc R.~Nickl and S.~Wang}, {\em On polynomial-time computation of
  high-dimensional posterior measures by {Langevin}-type algorithms}, arXiv
  preprint arXiv:2009.05298,  (2020).

\bibitem{raginsky2017non}
{\sc M.~Raginsky, A.~Rakhlin, and M.~Telgarsky}, {\em Non-convex learning via
  stochastic gradient {L}angevin dynamics: a nonasymptotic analysis}, in
  Conference on Learning Theory, PMLR, 2017, pp.~1674--1703.

\bibitem{robbins1951stochastic}
{\sc H.~Robbins and S.~Monro}, {\em A {S}tochastic {A}pproximation {M}ethod},
  The Annals of Mathematical Statistics, 22 (1951), pp.~400--407.

\bibitem{roberts2004general}
{\sc G.~O. Roberts and J.~S. Rosenthal}, {\em General state space {M}arkov
  chains and {MCMC} algorithms}, Probability Surveys, 1 (2004), pp.~20--71.

\bibitem{roberts1996exponential}
{\sc G.~O. Roberts and R.~L. Tweedie}, {\em Exponential convergence of
  {L}angevin distributions and their discrete approximations}, Bernoulli, 2
  (1996), pp.~341--363.

\bibitem{sexton1992hamiltonian}
{\sc J.~C. Sexton and D.~H. Weingarten}, {\em Hamiltonian evolution for the
  hybrid {Monte Carlo} algorithm}, Nuclear Physics B, 380 (1992), pp.~665--677.

\bibitem{shahbaba2014split}
{\sc B.~Shahbaba, S.~Lan, W.~O. Johnson, and R.~M. Neal}, {\em Split
  {Hamiltonian Monte Carlo}}, Statistics and Computing, 24 (2014),
  pp.~339--349.

\bibitem{smith2013sequential}
{\sc A.~Smith}, {\em Sequential {M}onte {C}arlo methods in practice}, Springer
  Science \& Business Media, 2013.

\bibitem{swendsen1986replica}
{\sc R.~H. Swendsen and J.-S. Wang}, {\em Replica {M}onte {C}arlo simulation of
  spin-glasses}, Physical Review Letters, 57 (1986), p.~2607.

\bibitem{tao2012topics}
{\sc T.~Tao}, {\em Topics in random matrix theory}, vol.~132, American
  Mathematical Soc., 2012.

\bibitem{wang2001determining}
{\sc F.~Wang and D.~P. Landau}, {\em Determining the density of states for
  classical statistical models: A random walk algorithm to produce a flat
  histogram}, Physical Review E, 64 (2001), p.~056101.

\bibitem{wang2001efficient}
\leavevmode\vrule height 2pt depth -1.6pt width 23pt, {\em Efficient,
  multiple-range random walk algorithm to calculate the density of states},
  Physical Review Letters, 86 (2001), p.~2050.

\bibitem{wanner1996solving}
{\sc G.~Wanner and E.~Hairer}, {\em Solving ordinary differential equations
  {II}}, vol.~375, Springer Berlin Heidelberg, 1996.

\bibitem{welling2011bayesian}
{\sc M.~Welling and Y.~W. Teh}, {\em Bayesian learning via stochastic gradient
  {L}angevin dynamics}, in Proceedings of the 28th International Conference on
  Machine Learning (ICML-11), 2011, pp.~681--688.

\bibitem{ye2021efficient}
{\sc X.~Ye and Z.~Zhou}, {\em Efficient sampling of thermal averages of
  interacting quantum particle systems with random batches}, The Journal of
  Chemical Physics, 154 (2021), p.~204106.

\end{thebibliography}

\end{document}